\documentclass[reqno,10pt]{amsart}

\usepackage[utf8]{inputenc}
\usepackage{amsmath,amsfonts,amssymb}
\usepackage{amsthm,upref}
\usepackage[a4paper]{geometry}
\usepackage{color} 
\usepackage[dvipsnames]{xcolor}
\usepackage{graphicx}
\usepackage{subfigure}
\usepackage{tabularx,float}
\usepackage[hidelinks]{hyperref}
\usepackage{comment} 
\usepackage{tabularx} 

\newtheorem{thm}{Theorem}[section]

\newtheorem{proposition}[thm]{Proposition}
\newtheorem{corollary}[thm]{Corollary}
\newtheorem{lemma}[thm]{Lemma}
\newtheorem{remark}[thm]{Remark}

\usepackage{todonotes}

\newcommand{\secref}[1]{Section~\ref{sec:#1}}

\newcommand{\seclab}[1]{\label{sec:#1}}
\newcommand{\eqlab}[1]{\label{eq:#1}}
\renewcommand{\eqref}[1]{(\ref{eq:#1})}

\newcommand{\figref}[1]{Fig.~\ref{fig:#1}}
\newcommand{\figlab}[1]{\label{fig:#1}}
\newcommand{\propref}[1]{Proposition~\ref{proposition:#1}}
\newcommand{\proplab}[1]{\label{proposition:#1}}

\newcommand{\lemmaref}[1]{Lemma~\ref{lemma:#1}}
\newcommand{\lemmalab}[1]{\label{lemma:#1}}
\newcommand{\remref}[1]{Remark~\ref{remark:#1}}
\newcommand{\remlab}[1]{\label{remark:#1}}
\newcommand{\thmref}[1]{Theorem~\ref{theorem:#1}}
\newcommand{\thmlab}[1]{\label{theorem:#1}}

\newcommand{\appref}[1]{Appendix~\ref{app:#1}}

\newcommand{\applab}[1]{\label{app:#1}}
\usepackage{enumitem}
\newcommand\response[1]{{\color{black}{#1}}}
\newcommand\responsenew[1]{{\color{black}{#1}}}
\usepackage{tikz}

\usepackage{amsaddr}
\title{The dud canard: Existence of strong canard cycles in $\mathbb R^3$} 
\author{K. Uldall Kristiansen}
\address{Department of Applied Mathematics and Computer Science, \\
Technical University of Denmark, \\
2800 Kgs. Lyngby, \\
Denmark }
\begin{document}
 
\maketitle
%
%
%
 \begin{abstract}
In this paper, we provide a rigorous description of the birth of canard limit cycles in slow-fast systems in $\mathbb R^3$ through the folded saddle-node of type II and the singular Hopf bifurcation. In particular, we prove  -- in the analytic case only -- that for all $0<\epsilon\ll 1$ there is a family of periodic orbits, born in the (singular) Hopf bifurcation and extending to $\mathcal O(1)$ cycles that follow the strong canard of the folded saddle-node. Our results can be seen as an extension of the canard explosion in $\mathbb R^2$, but in contrast to the planar case, the family of periodic orbits in $\mathbb R^3$ is not explosive. For this reason, we have chosen to call the phenomena in $\mathbb R^3$, the ``dud canard''. The main difficulty of the proof lies in connecting the Hopf cycles with the canard cycles, since these are described in different scalings. As in $\mathbb R^2$, we use blowup to overcome this, but we also have to compensate for the lack of uniformity near the Hopf bifurcation, due to its singular nature; it is a zero-Hopf bifurcation in the limit $\epsilon=0$. In the present paper, we do so by imposing analyticity of the vector-field. This allows us to prove existence of an invariant slow manifold, that is not normally hyperbolic. 
 \end{abstract}
\tableofcontents
\section{Introduction}\seclab{intro}
In this paper, we consider slow-fast systems of the form
\begin{equation}\eqlab{uvfast}
\begin{aligned}
 \dot u &=\epsilon U(u,v,\epsilon),\\
 \dot v&=V(u,v,\epsilon),
\end{aligned}
\end{equation}
with $U$ and $V$ sufficiently smooth and $0<\epsilon\ll 1$. Systems of this form occur in many different applications, including neuroscience \cite{amir2002a,izhi,rinzel1987a}, biology and chemical reaction networks \cite{epstein1996a,goryachev1997a} and many other areas, see  \cite{desroches2012a,kuehn2015} for further references. Geometric Singular Perturbation Theory \cite{jones_1995} is a collection of methods, based upon the ground-breaking work of Fenichel \cite{fen1,fen2,fen3}, that can be used to study systems of the form \eqref{uvfast} for $0<\epsilon\ll 1$. The point of departure for this theory, is the critical manifold 
\begin{align*}
 S= \{(u,v)\,:\, V(u,v,0)=0\},
\end{align*}
which is a set of equilibria for the associated layer problem:
\begin{equation}\eqlab{layeruv}
\begin{aligned}
 \dot u &=0,\\
 \dot v &=V(u,v,0),
\end{aligned}
\end{equation}
obtained by setting $\epsilon=0$ in \eqref{uvfast}. $S$ is said to be normally hyperbolic if all eigenvalues of $D_vV(u,v)$ for all $(u,v)\in S$ have nonzero real part. In particular, it is attracting (repelling) if the real part of these eigenvalues is negative (positive, respectively). Fenichel's theory, see  \cite{fen1,fen2,fen3,jones_1995}, then says that compact submanifolds $S_0$ of $S$ perturb to diffeomorphic locally invariant manifolds $S_\epsilon$ for all $0<\epsilon\ll 1$. The reduced flow on $S_\epsilon$ is to leading order given by the reduced problem:
\begin{align*}
 u' &= U(u,v,0),\\
 0&=V(u,v,0), 
\end{align*}
obtained by writing \eqref{uvfast} in terms of the slow time $\tau = \epsilon t$ with $()'=\frac{d}{d\tau}$ and subsequently letting $\epsilon\rightarrow 0$.
Moreover, stable and unstable manifolds of $S_0$ 
also perturb to $W^s(S_\epsilon)$ and $W^u(S_\epsilon)$, each having  invariant foliations by fibers; for full details see e.g. \cite{jones_1995}. 

Following work by Dumortier and Roussarie on the blowup method \cite{dumortier_1996}, there was in the early parts of the 2000s an effort \cite{krupa_extending_2001,krupa2001a,krupa2008a,szmolyan_canards_2001,wechselberger_existence_2005} to extend the geometric theory of Fenichel to points where normal hyperbolicity breaks down. The simplest type of breakdown, is perhaps observed in $\mathbb R^2$ and $\mathbb R^3$ with $S$ having folds, that divide the critical manifold into attracting and repelling subsets. In these cases, canard solutions are solutions of \eqref{uvfast} for $0<\epsilon\ll 1$ that -- counter-intuitively -- follow the attracting and repelling branches of the critical manifold by passing close to the fold. Canards are well-described in $\mathbb R^2$ and $\mathbb R^3$ \response{\cite{beno2001a,Benoit81,dumortier_1996,krupa2001a,krupa2008a,szmolyan_canards_2001,vonew}} and play an important role in applications and in the global dynamics of systems of the form \eqref{uvfast}, see e.g. \cite{desroches2012a,mujica2017a,Vo}.

In $\mathbb R^2$, canards of folded critical manifolds require an unfolding parameter \cite{krupa2001a} and here canard orbits may be limit cycles. In fact, the reference \cite{krupa2001a} proves that  there is a family of periodic orbits -- under some non-degeneracy conditions -- that include small (i.e. of size $o(1)$) Hopf cycles, intermediate cycles (see also \cite{dumortier2009a,huzak2019a}) and canard cycles of size $\mathcal O(1)$. Under some additional global properties, such family may be extended further to include canards with ``head'', see \cite{krupa2001a}, and eventually relaxation oscillations, as in the van der Pol system \cite{krupa2001a,pol1920a}. This situation is also known as the canard explosion \cite{bronsBarEli}, due to the fact that the canard limit cycles of different amplitude differ in parameter values by an order of $\mathcal O(e^{-c/\epsilon})$ for all $0<\epsilon\ll 1$, see full details in \cite{krupa2001a}.

In $\mathbb R^3$, on the other hand, with $u\in \mathbb R^2$ and $v\in \mathbb R$ in \eqref{uvfast}, canards of folded critical manifolds are generic, without parameters. They appear persistently at so-called folded singularities, which are singular points on the fold of a ``desingularized'' reduced problem, see \cite{beno2001a,szmolyan_canards_2001}. The folded singularities come in different generic versions: folded node, folded saddle and folded focus depending on the type of singularity, with only the former two producing canard solutions. 

The folded node is of particular interest due to its connection to mixed-mode oscillations. In summary, the folded node gives rise to a weak canard (under a nonresonance condition) and a strong canard -- essentially due to the weak and the strong directions of the linearization of the node -- and close to folded singularity, it has been shown, using blowup \cite{szmolyan_canards_2001}, that the tangent space of the attracting slow manifold twists a finite number of times along the weak canard. This implies, due to the contractivity towards the weak canard on the attracting side of the critical manifold, that an open set of points twists upon passage through the folded node. Upon composition with a global return mapping, this provides a simple mechanism for producing attracting limit cycles, see \cite{brons-krupa-wechselberger2006:mixed-mode-oscil}, that are of mixed-mode type, see also \cite{desroches2012a}.

The folded saddle-node is a bifurcation of the folded singularity. It comes in different types I and II, \response{see \cite{krupa2008a,vonew}}, but the unfolding of type II  -- at the level of the reduced problem -- produces a transcritical bifurcation of a true singularity and a folded one. In this paper, we will only focus on the type II and we will therefore continue to refer to this case simply as the folded saddle-node. This bifurcation is known to give rise to a Hopf bifurcation \cite{krupa2008a}. It is a singular Hopf bifurcation \cite{braaksma1998a,guckenheimer2008a} due to the fact that the linearization (upon blowup) has eigenvalues of the form $\sim \pm i \omega,\sim \epsilon\lambda$, $\omega,\lambda\ne 0$ as $\epsilon\rightarrow 0$ at the Hopf bifurcation; it is therefore a zero-Hopf bifurcation \cite{baldom2019a,baldom2008a} for $\epsilon=0$.

The interest in the folded saddle-node comes from the fact that it marks the onset (or termination) of mixed-mode oscillations through the folded node. However, in the author's opinion, the details of this onset/termination and the connection of mixed-mode oscillations with the Hopf cycles is still not fully understood. \cite[Theorem 4.2]{brons-krupa-wechselberger2006:mixed-mode-oscil} relates to the connection problem, but only indirectly. Specifically, the bifurcation described in this theorem -- where a return mechanism transverses the strong canard -- does not relate to the Hopf bifurcation. The reference \cite{mujica2017a} is another interesting study, based upon detailed numerical computations. Here the Hopf cycles are continued using the software package AUTO and it is demonstrated (for a fixed small value of $\epsilon>0$) that these cycles are of relaxation type without mixed-modes. In particular, the periodic orbits of mixed-mode type in the model system of \cite{mujica2017a} form isolas that are disconnected from the branch of Hopf cycles (which undergo period doubling bifurcations). 
At the same time, \cite{guckenheimer2008a} studies a normal form for a (different) singular Hopf with two slow variables, computing the Lyapunov coefficient and demonstrating additional bifurcations (periodic doubling and torus) using numerical computations along the branch of period orbits that appear from the Hopf bifurcation.  \cite{kukmgp} describes the onset of mixed-modes in a cusped saddle-node in a system with symmetry.

\responsenew{Finally, \cite{zaks2011a} studies a Fitz-Hugh-Nagumo-like system and demonstrates, also through numerical computations, that the Hopf cycles lose stability via a sequence of period-doubling bifurcations. Interestingly, the cascade follows the Feigenbaum constant for conservative systems for small values of $\epsilon>0$. To the best of the author's knowledge, these results on period doubling bifurcations have not been studied rigorously.}

It has also been speculated, following work on the Koper model \cite{koper1991a,koper1992a}, that the folded saddle-node may be associated with homoclinics and Shilnikov bifurcations. This was demonstrated for the Koper model in \cite{guckenheimer2015a} using sophisticated numerical methods, among other things.
At the same time, it is by now known \cite{baldom2019a,broer1984a} that generic unfoldings of the zero-Hopf bifurcation produce Shilnikov bifurcations. In future work, the present author hopes to pursue these bifurcations (period doubling and Shilnikov) rigorously in the context of the folded saddle-node.  In preparation, we will in this paper extend the results of \cite{krupa2001a} on the family of periodic canard orbits in $\mathbb R^2$ to the $\mathbb R^3$-context. Whereas the family of canard cycles have ``explosive growth'' in the planar context,  the growth rate is regular in $\mathbb R^3$, with canard cycles of different $\mathcal O(1)$-amplitude corresponding (in general)  to parameter values that differ by an $\mathcal O(1)$-amount. 
For this reason, we have chosen to call the phenomena we describe as the ``dud canard'' instead of the canard explosion.

\subsection{Setting}
We consider the following normal form for the folded node/folded saddle-node \cite{krupa2008a}:
\begin{equation}\eqlab{fast}
\begin{aligned}
 \dot x &=\epsilon (y-(\mu+1)z+F(x,y,z,\epsilon,\mu)),\\
  \dot y &=\epsilon (\frac12 \mu+ G(x,y,z,\epsilon,\mu)),\\
 \dot z &= x+z^2 + z H(x,y,z,\epsilon,\mu),
\end{aligned}
\end{equation}
in the regime $\epsilon>0$, $\epsilon\sim 0$ and $\mu\sim 0$,
where $F$, $G$ and $H$ are smooth and higher order in the following sense:
\begin{align*}
 F(x,y,z,\epsilon,\mu)=\mathcal O(x,\epsilon,(\vert y\vert+\vert z\vert)^2),\quad G(x,y,z,\epsilon,\mu)=\mathcal O(x,y,z,\epsilon),
\end{align*}
and
\begin{align*}
 H(x,y,z,\epsilon,\mu)=\mathcal O(xz,xy,z^2,\epsilon).
\end{align*}
\response{In comparison with \cite{krupa2008a} we have $z H$ instead of $H$. This plays little role, but the latter can be brought into the former by a transformation of $x$.} $\mu$ will be our bifurcation parameter. 

The system is normalized such that the following holds:\response{
\begin{lemma}\lemmalab{mprop}
The layer problem has a critical manifold $S$ of the graph form $x=m(y,z,\mu)$ with $m$ satisfying 
\begin{align}\eqlab{mprop}
m(y,0,\mu)=\frac{\partial}{\partial z}m(y,0,\mu)\equiv 0,\quad \frac{\partial^2}{\partial z^2}m(0,0,\mu)=-2.
\end{align} Locally, the manifold $S$ is normally hyperbolic for $z\ne 0$, being  attracting along $z<0$ and repelling for $z>0$. The line defined by $(0,y,0)$ is a fold line of $S$.
\end{lemma}
\begin{proof}
 The critical manifold is given by 
 \begin{align*}
  x+z^2 + z H(x,y,z,0,\mu)=0.
 \end{align*}
We can solve this equation for $x$ by the implicit function theorem $x=m(y,z,\mu)$ and a simple calculation shows \eqref{mprop}.
\end{proof}
}
It follows that $m$ takes the form $m(y,z,\mu)=-z^2(1+\mathcal O(y,z))$, and we can therefore define smooth functions \response{$L_1$ and $L_2$ by 
\begin{align*}
 \frac{\partial m}{\partial y}(y,z,\mu) =: z^2L_1(y,z,\mu),\quad
 \frac{\partial m}{\partial z}(y,z,\mu) =: -2z(1+L_2(y,z,\mu)).
\end{align*}
Clearly, $L_2(0,0,\mu)\equiv 0$.}
 Upon desingularization, \response{corresponding to multiplication of the right hand side by $\frac{\partial m}{\partial z}(y,z,\mu)$}, the reduced problem takes the following form:
\begin{equation}\eqlab{reduced}
\begin{aligned}
 y' &=-2z(1+L_2(y,z,\mu)) \left(\frac12 \mu+G(m(y,z),y,z,0,\mu)\right),\\
 z' &=y-(\mu+1)z+F(m(y,z),y,z,0,\mu)\responsenew{-z^2 L_1(y,z,\mu) \left(\frac12 \mu+G(m(y,z),y,z,0,\mu)\right)},
\end{aligned}
\end{equation}
with $x=m(y,z)$. Consequently, we have a folded singularity $p$ \cite{szmolyan_canards_2001} at $(x,y,z)=(0,0,0)$. In fact, $(y,z)=(0,0)$ is partially hyperbolic for \eqref{reduced} with $\mu=0$ and a center manifold reduction shows that \eqref{reduced} undergoes a transcritical bifurcation for $\mu=0$ if 
\begin{align}
 \lambda:=\partial_y G(\textbf 0)+\partial_z G(\textbf 0)\ne 0.\eqlab{assumption}
\end{align}
It is this bifurcation that is known as a folded saddle-node (of type II \cite{krupa2008a}) for the slow-fast system \eqref{fast}. We illustrate the bifurcation in \figref{fig:fsn} in terms of the slow time; in comparison with \eqref{reduced} the directions on the repelling sheet are therefore reversed \cite{szmolyan_canards_2001}. Here \figref{fig:fsn} (a) shows $\lambda<0$ whereas \figref{fig:fsn} (b) shows $\lambda>0$.

\begin{figure}[h!]
\begin{center}
\subfigure[$\lambda<0$]{\includegraphics[width=.93\textwidth]{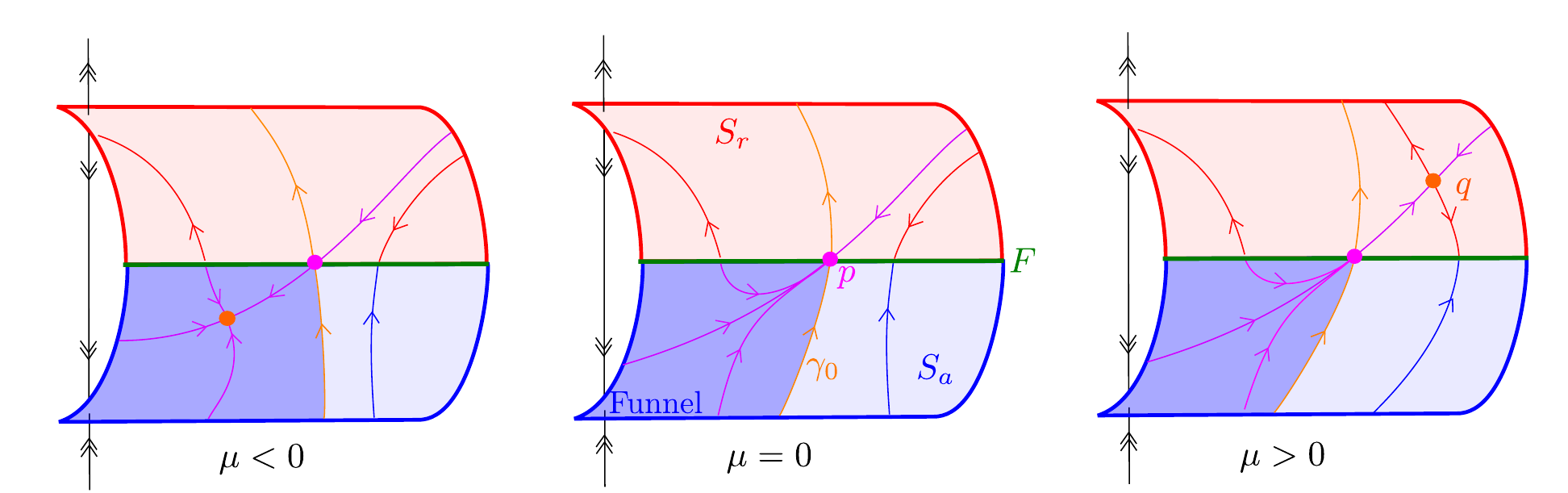}}
\subfigure[$\lambda>0$]{\includegraphics[width=.93\textwidth]{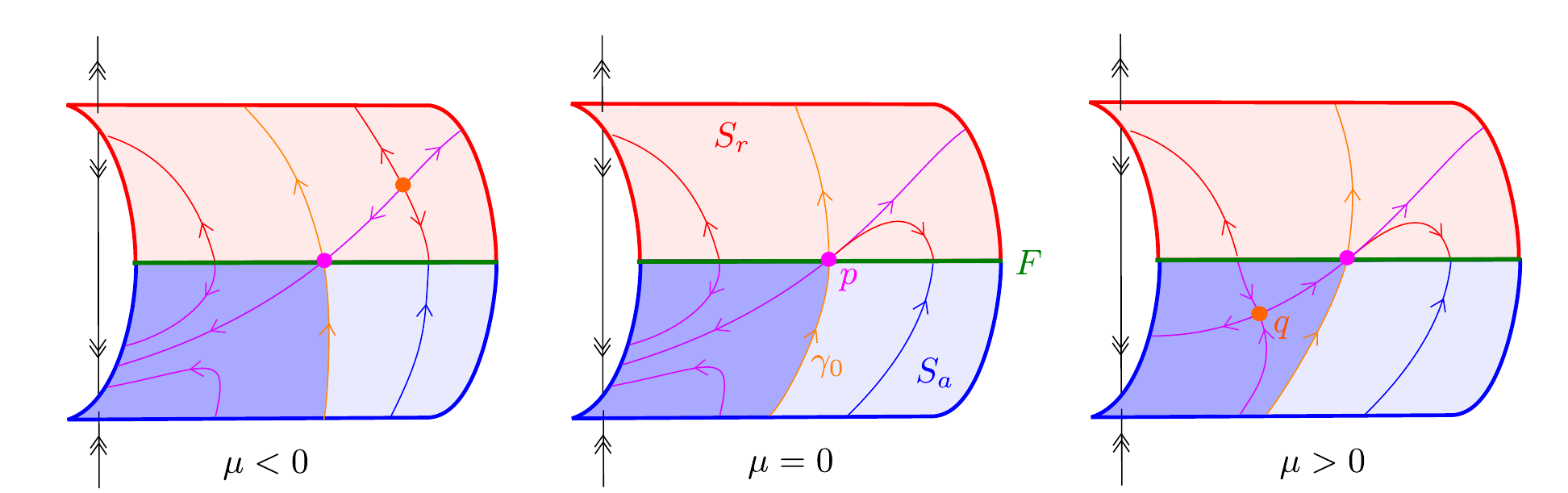}}
\end{center}
\caption{Illustration of the folded saddle-node of type II. It is a transcritical bifurcation of the desingularized reduced problem \eqref{reduced}. The dynamics illustrated here is in terms of the slow time. Consequently, in comparison with \eqref{reduced}, we have reversed the directions on the repelling sheet. The sign of $\lambda$ (see \eqref{assumption}) determines on what side of the two sheets the saddle point $q$ appears. For $\lambda<0$, shown in (a), $q$ lies on the repelling sheet $S_r$, appearing for $\mu\gtrsim 0$. For $\lambda>0$, shown in (b), $q$ lies on the the attracting sheet $S_a$, appearing for $\mu\lesssim 0$. In both cases, the folded singularity $p$ is a folded node on one side of $\mu$ and folded saddle on the other side. The strong singular canard $\gamma_0(\mu)$ divides $S_a$ into separate sets. In particular, for $\mu\ge 0$ and $\lambda<0$, the dark shaded region in (a) is a funnel region. Here all points pass through $p$ upon following the reduced flow. }
\figlab{fig:fsn}
\end{figure}


On the other hand, for each $\mu\sim 0$, there exists a strong stable manifold $\gamma_0(\mu)$ of $(y,z)=0$ for \eqref{reduced}, known as the strong singular canard in the $(x,y,z)$-space (orange in \figref{fig:fsn}). It is well-known \cite{szmolyan_canards_2001} that $\gamma_0(\mu)$ persists as a (maximal) canard $\gamma_{\epsilon}(\mu)$ connecting fixed copies of Fenichel slow manifolds $S_{a,\epsilon}$ and $S_{r,\epsilon}$ as perturbations of (appropriate) compact subsets $S_{a,0}$ and $S_{r,0}$ of $S_a:=S\cap \{z<0\}$ respectively $S_r:=S\cap \{z>0\}$ for all $0<\epsilon\ll 1$.

\subsection{Main result}\seclab{main}
In this paper, we are interested in canard cycles, i.e. periodic orbits that follow $\gamma_0(\mu)$ on $S$. In particular, under the assumption \eqref{assumption} and analyticity of $F$, $G$ and $H$, we prove the existence of a family of periodic orbits. 
 \responsenew{These periodic orbits are singular perturbations of a set of singular canard cycles, that we describe in the following lemma and illustrate in \figref{fig:Gamma0} (using the viewpoint in \figref{fig:fsn} and a projection onto the $(x,z)$-plane). 
\begin{lemma}\lemmalab{lemma0}
 There is an $h_1>0$ sufficiently small and a smooth function $\overline \mu_0:[0,h_1]\rightarrow \mathbb R$ such that for each $h\in (0,h_1]$, there is a singular canard cycle $\Gamma_{0,h}$ that intersects $z=0$ in $(x,y,z)=(-h^2,*,0)$ and is the union of (a) a segment of the orbit $\gamma_0(\overline \mu_0(h))$ of \eqref{reduced}$_{\mu=\overline\mu_0(h)}$, connecting $S_a$ and $S_r$ through $p$, and (b) a fast jump of \eqref{fast}$_{\epsilon=0,\mu=\overline \mu_0(h)}$.
\end{lemma}
\begin{proof}
See \secref{lemma0}.
\end{proof}
}

\begin{figure}[h!]
\begin{center}
{\includegraphics[width=.55\textwidth]{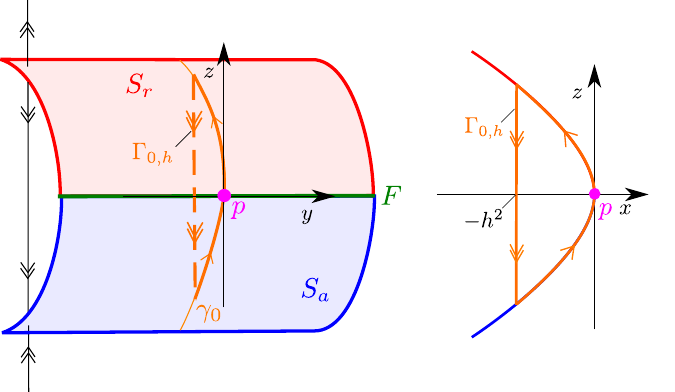}}
\end{center}
\caption{\responsenew{The singular cycle $\Gamma_{0,h}$, $h\in (0,h_1]$, using the viewpoint in \figref{fig:fsn} to the left and a projection onto the $(x,z)$-plane to the right, see \lemmaref{lemma0}}. }
\figlab{fig:Gamma0}
\end{figure}

\begin{thm}\thmlab{main}
\response{Consider \eqref{fast} as a normal form for the folded saddle-node of type II and suppose (a) that  \eqref{assumption} holds and (b) that $F$, $G$ and $H$ are analytic functions in the phase space variables $x$, $y$, $z$ with smooth dependency on $\epsilon$ and $\mu$.} Then there exists an $h_1>0$ such that for all $0<\epsilon\ll 1$ the following holds: There is a Hopf bifucation of \eqref{fast} at $\mu=\mu_H(\sqrt \epsilon)$, $\mu_H(0)=0$, and a family of periodic orbits $\Gamma_{\epsilon,h}$, \response{$h\in (0,h_1]$}, along $\mu=\overline \mu(\epsilon,h)$, where $\overline \mu$ is continuous and satisfies $\overline \mu(\epsilon,0)\equiv \mu_H(\sqrt \epsilon)$ \responsenew{and $\overline \mu(0,h)\equiv\overline \mu_0(h)$}. Moreover, for each $h\in (0,h_1]$ fixed, \responsenew{$\Gamma_{\epsilon,h}$ converge in the Haussdorff distance as $\epsilon\rightarrow 0$ to the singular canard cycle $\Gamma_{0,h}$.} The limit is uniform on compact subsets of $(0,h_1]$. 
\end{thm}

                           We do not aim to describe the stability of the cycles. This requires more work in general, but see \remref{stability} and \secref{discuss} below. 
To prove \thmref{main}, we  apply the following blowup transformation to the extended system (\eqref{fast},$\dot \epsilon=0$) of the folded singularity:
\begin{align*}
 (r,(\bar x,\bar y,\bar z,\bar \epsilon))\mapsto \begin{cases}
                                                  x &=r^2 \bar x,\\
                                                  y&=r\bar y,\\
                                                  z &=r\bar z,\\
                                                  \epsilon &=r^2\bar \epsilon.
                                                 \end{cases}
\end{align*}
for $r\ge 0$, $(\bar x,\bar y,\bar z,\bar \epsilon)\in S^3$ and $\mu\sim 0$. Here we follow the blowup used in \cite{szmolyan_canards_2001}. In particular, in contrast to \cite{krupa2008a}, we will \responsenew{not include $\mu$ in the transformation and consequently do not blowup $\mu = 0$} (at this stage, at least). We use two separate charts $\bar x=-1$ and $\bar \epsilon=1$ with chart-specific coordinates $ (\epsilon_1,r_1,y_1,z_1)$ and $ (r_2,x_2,y_2,z_2)$ defined by :
\begin{align}
 (\epsilon_1,r_1,y_1,z_1) \mapsto \begin{cases}
                                                  x &=-r_1^2,\\
                                                  y&=r_1y_1,\\
                                                  z &=r_1 z_1,\\
                                                  \epsilon &=r_1^2 \epsilon_1,
                                                 \end{cases}\eqlab{barxneg1}\\
    (r_2,x_2,y_2,z_2) \mapsto \begin{cases}
                                                  x &=r_2^2x_2,\\
                                                  y&=r_2y_2,\\
                                                  z &=r_2 z_2,\\
                                                  \epsilon &=r_2^2,
                                                 \end{cases}                                             \eqlab{bareps1}
\end{align}
respectively. The charts overlap on $\bar x<0$ and here the change of coordinates are given by the expressions:
\begin{equation}\eqlab{cc12}
\begin{aligned}
 r_2 & = r_1 \sqrt{\epsilon_1},\\
 x_2 &= -\epsilon_1^{-1},\\
 y_2 &= y_1/\sqrt{\epsilon_1},\\
 z_2 &= z_1/\sqrt{\epsilon_1},
\end{aligned}
\end{equation}
for $\epsilon_1>0$. 


In \cite{krupa2008a}, the authors also describe the Hopf bifurcation in the $\bar \epsilon=1$-chart. The new contribution of \thmref{main} is that we describe the family of periodic orbits bifurcating from the Hopf bifurcation in a full (i.e. $\epsilon$-independent) neighborhood of the folded saddle-node for all $0<\epsilon\ll 1$. This family includes periodic orbits with amplitude $h=\mathcal O(1)$ \responsenew{as perturbations of $\Gamma_{0,h}$ for all $0<\epsilon\ll 1$}.


Our approach is similar to the approach for the analysis of the canard explosion in the planar case, see \cite{krupa2001a}. The paper \cite{krupa2001a} also uses a Melnikov approach to extend the Hopf cycles in the associated scaling chart and then subsequently extend these to canard cycles by working in directional charts. The latter connection problem is already complicated in \cite{krupa2001a}. In the present paper, we feel that our proof in $\mathbb R^3$ is relatively simple. It basically extends the classical way of obtaining canard cycles in $\mathbb R^2$, by flowing points forward and backward along the attracting and repelling sheets and then extending this close to the folded saddle-node through blowup (using the chart $\bar x=-1$). Shilnikov variables \cite{deng1988a} and normal forms \cite{Guckenheimer97,ilyashenko1991a} are used to study the necessary transition maps. This approach could also be used as an alternative way \response{(that is potentially simpler)} to solve the connection problem in $\mathbb R^2$, although we have not attempted to do so. 

In contrast, the problem of connecting the Hopf cycles with the ones obtained by the Melnikov analysis is more complicated here in $\mathbb R^3$ than in the $\mathbb R^2$-context of \cite{krupa2001a}.
This is also related to our assumption on analyticity of $F$, $G$ and $H$ in \thmref{main}, which may seem unusual for results in this direction. To explain the difficulty, we recall from \cite{krupa2008a} that there is a one-dimensional critical manifold $C_2$ in the $\bar \epsilon=1$-chart for $\mu=r_2=0$ given by \responsenew{the graph} $x_2=-y_2^2,z_2=y_2$, \responsenew{over} $y_2\in \mathbb R$. The linearization \responsenew{around $C_2$} has imaginary eigenvalues at $y_2=0$ and the reduced problem has a hyperbolic equilibrium precisely at this point; this is what produces the Hopf bifurcation for all $0<r_2\ll 1$. However, it is nontrivial to study the Hopf cycles in a fixed (small) neighborhood of the Hopf bifurcation, since the eigenvalues are of the form $\pm (1+\mathcal O(r_2)) i, r_2 (\lambda+\mathcal O(r_2))$, i.e. a zero-Hopf bifurcation occurs at $\mu=r_2=0$. As we see it, there are two ways to perform the analysis of the Hopf cycles: (i) Perform a center manifold reduction and apply the Hopf bifurcation theorem there. Or: (ii) Straighten out the (strong) unstable/stable manifold ($\lambda\gtrless 0$, respectively), introduce polar coordinates in the transverse direction and apply Melnikov-like methods to construct the periodic orbits as fixed-points of a return map. However, both approaches are not uniform (at least in the smooth setting) with respect to $\epsilon\rightarrow 0$, due to the fact that the eigenvalue with the nonzero real part $\sim \epsilon\lambda$, providing the necessary hyperbolicity, goes to zero as $\epsilon\rightarrow 0$. 
Nevertheless, in the analytic case, we can extend the proof of the unstable/stable manifold to a fixed neighborhood (by following \cite[Section 3]{de2020a}), and this allows us to apply the approach (ii).  There is no clear way to obtain this result in the smooth setting,  since the standard proof of the unstable/stable manifold rests (more directly) upon exponential estimates (that are nonuniform in the present context). Having said that,  \cite{bonckaert1986a,bonckaert2005a} both study a zero-Hopf bifurcation (in the case of \cite{bonckaert2005a}, also at a parameter value $\epsilon=0$) in the $C^\infty$-setting and prove existence of a one-dimensional invariant curve tangent to the zero eigenspace. However, these results do not directly apply in our setting and the methods would have to be modified. For example, in \cite{bonckaert1986a} there is no parameter. Moreover, to bring the folded saddle-node near the Hopf bifurcation into the normal form in \cite{bonckaert2005a} we would have to perform an $\epsilon$-dependent scaling of the variables; notice in particular, that \cite[Equation (1)]{bonckaert2005a} does not depend upon the slow variable $z$ for $\epsilon=0$. We therefore leave the extension to $C^\infty$ (and potentially even to $C^k$) to future work. 

\subsection{Overview}\seclab{overview}
In the remainder of the paper, we work to prove \thmref{main}. First in \secref{small}, we describe the ``small periodic orbits'', extending all the way down to the Hopf bifurcation. \response{The reason for referring to these orbits of \eqref{fast} as small is that their amplitude is $o(1)$ with respect $\epsilon\rightarrow 0$.} Subsequently, in \secref{inter}, we describe the ``intermediate orbits'' that connect the small periodic orbits with the canard orbits of amplitude $O(1)$.  In \secref{completing}, we complete the proof of the theorem. Finally, in \secref{discuss} we conclude the paper through a discussion of the results and potential future work. 

\section{Existence of small periodic orbits}\seclab{small}
\response{
In the following, we first (see \secref{eqn2}) revisit the most basic results (including the Hopf bifurcation) of the system: 
%
\begin{equation}\eqlab{eqn2}
\begin{aligned}
 \dot x_2 &= y_2-(\mu+1)z_2 + r_2 F_2(x_2,y_2,z_2,r_2,\mu),\\
 \dot y_2 &= \frac12 \mu +r_2 a_1 y_2 + r_2 a_2 z_2 + r_2^2 \overline G_2(x_2,y_2,z_2,r_2,\mu),\\
 \dot z_2 &=x_2 + z_2^2+r_2 z_2 H_2(x_2,y_2,z_2,r_2,\mu).
\end{aligned}
\end{equation}
(The results of \secref{eqn2} can also be found in  \cite{krupa2008a,szmolyan_canards_2001}.)
The system \eqref{eqn2} is obtained from writing \eqref{fast} in the scaled coordinates $(x_2,y_2,z_2)$ defined by \eqref{bareps1}, and using a desingularization through division of the right hand side by $r_2=\sqrt{\epsilon}$. 
Here $F_2$, $\overline G_2$ and $H_2$ are smooth, in particular analytic in the space variables $x_2,y_2$ and $z_2$. As in \cite{krupa2008a}, we have put $a_1=\partial_y G(\textnormal{\textbf{0}})$, $a_2=\partial_z G(\textnormal{\textbf{0}})$. Recall then that 
\begin{align*}
 \lambda:=a_1+a_2\ne 0,
\end{align*}
by assumption \eqref{assumption}.
}


\response{Subsequently in \secref{existenceslow}, following the strategy (ii) (described in the last paragraph of \secref{intro}) for studying the Hopf bifurcation, we prove the existence (using \cite[Section 3]{de2020a}) of a slow manifold $Z_{2,r_2}$, see \propref{slowmanifold}, that is (a) analytic in the space variables and smooth in $\epsilon$ and $\mu_2$, and (b) a perturbation of a compact submanifold of the critical manifold 
\begin{align}\eqlab{C2set}
C_2=\{(x_2,y_2,z_2) \,:\, x_2=-y_2^2,z_2=y_2, \,y_2\in \mathbb R\}.
\end{align}
of \eqref{eqn2} for $r_2=\mu=0$, see \lemmaref{C2}. By straightening out $Z_{2,r_2}$, we ensure that polar coordinates in the normal directions are well-defined. This allows us to set up a return map, defined in a full neighborhood, which we describe using Melnikov theory in \secref{Melnikov}.} 
\response{The family of periodic orbits of \eqref{eqn2}, that we obtain in this way, become periodic orbits of \eqref{fast} (upon \textit{blowing down} using \eqref{bareps1}) with amplitudes of order $o(1)$ with respect to $\epsilon\rightarrow 0$ (hence: ``{small periodic orbits}''.}


\subsection{Analysis of \eqref{eqn2}}\seclab{eqn2}
For $\mu\sim 0$, $r_2\sim 0$ this system is slow-fast with $x_2$ and $z_2$ being fast and $y_2$ slow. In particular, $\mu=r_2=0$ gives  an associated layer problem
\begin{equation}\eqlab{layer2}
\begin{aligned}
 \dot x_2 &= y_2-z_2,\\
 \dot y_2 &=0,\\
 \dot z_2 &=x_2+z_2^2.
\end{aligned}
\end{equation}
\begin{lemma}\lemmalab{C2}
  \response{The set $C_2$ in \eqref{C2set} is a critical manifold of \eqref{layer2}} and it is normally attracting for $y_2<0$ and normally repelling for $y_2>0$. 
 
 On the other hand, $y_2=0$ is degenerate for $C_2$ with the linearization around $(0,0,0)$ having two imaginary eigenvalues $\pm i$. 
 In particular, \eqref{layer2} is time-reversible within $y_2=0$ and the orbit 
 \begin{equation}\eqlab{gamma20}
\begin{aligned}
 \gamma_{2,0}(0): \begin{cases}
            x_2(t_2)& = -\frac14 t_2^2+\frac12,\\
            z_2(t_2) &=\frac12 t_2,
           \end{cases}\quad t_2\in \mathbb R,
 \end{aligned}
 \end{equation}
 is a separatrix in the $(x_2,z_2)$-plane, separating closed periodic orbits $\phi_h(t)=(x_{2h}(t),z_{2h}(t))$,  $x_{2h}(0)=-h,z_{2h}(0)=0$ with period $T_0(h)>0$ for any $h>0$, from unbounded orbits. 

\end{lemma}

\begin{proof}
The statement regarding the stability of $C_2$ follows from simple calculations. The analysis for $y_2=0$ is also straightforward, see \cite{krupa2001a}. 
\end{proof}
We illustrate the dynamics in the $(x_2,z_2)$-plane in \figref{fig:x2z2}. 
\begin{figure}[h!]
\begin{center}
{\includegraphics[width=.53\textwidth]{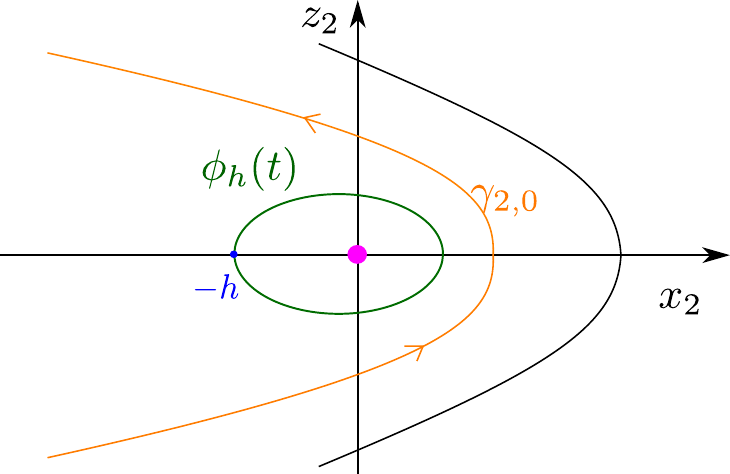}}
\end{center}
\caption{The layer problem \eqref{layer2} in the $(x_2,z_2)$-plane. The canard orbit $\gamma_{2,0}$ (orange) separates bounded periodic orbits $\phi_h$ (green) from unbounded orbits. }
\figlab{fig:x2z2}
\end{figure}
The solution \eqref{gamma20} for $\mu=0$ belongs to a $\mu$-family of solutions 
\begin{equation}\eqlab{gamma2}
\begin{aligned}
 \gamma_{2,0}(\mu): \begin{cases}
            x_2(t_2)& = -\frac14 t_2^2+\frac12,\\
            y_2(t_2) &=\frac{\mu}{2}t_2 ,\\
            z_2(t_2) &=\frac12 t_2,
           \end{cases}\quad t_2\in \mathbb R,
 \end{aligned}
 \end{equation}
 of \eqref{eqn2} for $r_2=0$. It corresponds to the blowup of the strong canard $\gamma_0(\mu)$, see \cite{szmolyan_canards_2001}.

To describe the reduced problem on $C_2$, we consider the scaling
\begin{align}
 \mu = r_2 \mu_2.\eqlab{mu2scaling}
\end{align}
Then with $\mu_2$ fixed,  $r_2=0$ implies $\mu=0$.
\cite{krupa2008a} also uses this scaling, but it is strictly speaking not necessary for the present analysis. In fact, it will be crucial to our approach not to scale $\mu$ for the description of the intermediate periodic orbits that we study later in the paper. 

With \eqref{mu2scaling}, we obtain the following reduced problem on $C_2$:
\begin{align*}
 y_2' &= \frac12 \mu_2+\lambda y_2,
\end{align*}
in terms of a slow time.
For $\lambda\ne 0$, we have a hyperbolic equilibrium at 
\begin{align}\eqlab{y2eq}
y_2= -\frac{\mu_2}{2\lambda}.
\end{align}
Consequently, for $\mu_2=0$ this equilibrium lies at $y_2=0$, corresponding to the degenerate point of $C_2$. This gives rise to the (singular) Hopf bifurcation \cite{guckenheimer2008a}. 

\subsection{Existence of an analytic slow manifold}\seclab{existenceslow}
For any $\nu>0$, $k\in \mathbb N$, let 
\begin{align*}
 \Omega^k(\nu):=\left\{h=\sum_{n=0}^\infty h_n (\cdot)^n :\overline{B_\nu}\rightarrow \mathbb C^k \,\vert \, h_n\in \mathbb R^k,\,\response{\Vert h\Vert_{\Omega^k}} <\infty\right\},
\end{align*}
with $\overline{B_\nu}$ being the closure of the open disc $B_\nu\subset \mathbb C$ of radius $\nu$ centered at the origin, and where
\begin{align}
 \Vert \sum_{n=0}^\infty h_n (\cdot)^n \Vert_{\Omega^k}:=\sum_{n=0}^\infty \vert h_n\vert \nu^n.\eqlab{Omega1Norm}
\end{align}
In other words, $\Omega^k(\nu)$ consists of absolutely convergent power series on $\overline{B_\nu}$ with real coefficients. Consequently, if $h\in \Omega^k(\nu)$ then the resulting function $h:B_\nu\rightarrow \mathbb C^k$ is real analytic \response{(here real refers to the fact that $h(v)$ is real when $v$ is so)}. Moreover, any real analytic function, having \response{a real and} absolutely convergent power series at the origin with radius of convergence $\rho>\nu$, belongs to $\Omega^k(\nu)$. 

\begin{lemma}
\response{$(\Omega^k(\nu),\Vert \cdot \Vert_{\Omega^k})$} is a Banach space.
\end{lemma}
\begin{proof}
Suppose that $h^m=\sum_{n=0}^\infty h_n^m (\cdot)^n\in \Omega^k(\nu)$ is a Cauchy sequence:
\response{
\begin{align*}
 \Vert h^{m_k}-h^{m_j}\Vert_{\Omega^k} = \sum_{n=0}^\infty \vert h_n^{m_k}-h_n^{m_j}\vert \nu^n\rightarrow \infty,
\end{align*}using the definition of the norm \eqref{Omega1Norm},
for $m_k,m_j\rightarrow \infty$. It follows that $\{h_n^m \nu^n\}_{n=0}^\infty$ is a Cauchy sequence in the complete sequence space $l_1$}, which therefore converges: \response{$\lim_{m\rightarrow \infty} \{h_n^m\nu^n\}_{n=0}^\infty=:\{h_n^\infty\nu^n \}_{n=0}^\infty$. Let $h^\infty:=\sum_{n=0}^\infty h_n^\infty (\cdot)^n\in \Omega^k(\nu)$. Then 
\begin{align*}
 \lim_{m\rightarrow \infty}\Vert h^m -h^\infty \Vert_{\Omega^k} =\lim_{m\rightarrow \infty}\sum_{n=1}^\infty \vert h^m_n-h^\infty_n\vert \nu^n =0,
\end{align*}
and $h^m=\sum_{n=0}^\infty h_n^m (\cdot)^n$ therefore converges to $h^\infty$. Consequently, $\Vert \cdot \Vert_{\Omega^k}$ is a Banach norm on $\Omega^k(\nu)$ as claimed.}
\end{proof}


In the following, it will be \response{convenient} to write $u:=(x_2,z_2)\in \mathbb R^2$.
\begin{proposition}\proplab{slowmanifold}
Consider \eqref{eqn2} with \eqref{mu2scaling} and $F_2$, $G_2$ and $H_2$ all smooth functions, specifically analytic in $x_2,y_2,z_2$.
Fix $k\in \mathbb N$. Then there are constants $\delta>0$ and $\nu>0$, both sufficiently small, such that the following holds: There exists a locally invariant one-dimensional manifold $Z_{2,r_2}$ of \eqref{eqn2} of the graph form:
\begin{align*}
u = m(y_2,r_2,\mu_2),
\end{align*}
where $(r_2,\mu_2)\mapsto m(\cdot,r_2,\mu_2)\in \Omega^2(\nu)$ is $C^k$ for $0\le r_2\ll 1$, $\mu_2\in [-\delta,\delta]$, satisfying $m(y_2,0,\mu_2)=(-y_2^2,y_2)$ such that $Z_{2,0}$ is a submanifold of $C_2$.

\end{proposition}
\begin{proof}

 The proof follows \cite{de2020a}, see specifically \cite[Lemma 3.25]{de2020a}. Although this lemma addresses existence of a fixed point in the Borel plane, we can adapt the overall construction. Moreover, many of the estimates we derive draw inspiration from \cite{de2020a}. 

 Firstly, seeing that the linearization around $C_2$ has two eigenvalues that are bounded away from $0\in \mathbb C$, it is standard (see e.g. \cite{de2020a,kristiansenwulff}) that there is a formally invariant manifold 
\begin{align*}
 u = \sum_{n=0}^\infty r_2^n h_n(y_2,\mu_2),
\end{align*}
with the right hand side to be understood as a formal power series. Each $h_n$ is analytic in $y_2$, smooth in $\mu_2$. Specifically, $h_0(y_2,\mu_2)= (-y_2^2,y_2)$. Now define $\tilde u$ by 
\begin{align*}
 u = h_0(y_2,\mu_2)+r_2 h_1(y_2,\mu_2) + \tilde u,
\end{align*}
Then upon dividing the right hand side by $\lambda$ and setting $$v=y_2+\frac{1}{2\lambda}\mu_2,$$ we obtain
%
%
 \begin{equation}\nonumber
 \begin{aligned}
  \dot{\tilde u} &= \left(A(v,\mu_2)+\mathcal O(r_2,\tilde u)\right)\tilde u + \mathcal O(r_2^{2}),\\
  \dot v &=r_2 \left(v + \mathcal O(r_2,\tilde u)\right),
 \end{aligned}
 \end{equation}
where
\begin{align*}
 A(v,\mu_2) = \lambda^{-1}\begin{pmatrix}
                  0 & -1\\
                  1 & 2v-\frac{1}{\lambda}\mu_2 
                 \end{pmatrix},
\end{align*}
Now drop the tilde on $u$ and let 
\begin{align*}
 u = r_2 \tilde u. 
\end{align*}
This gives 
 \begin{equation}\eqlab{tildeuv}
 \begin{aligned}
  \dot{\tilde u} &= A(v,\mu_2)\tilde u + r_2 R_2(\tilde u,v,r_2,\mu_2) ,\\
  \dot v &=r_2 \left(v + r_2 P_2(\tilde u,v,r_2,\mu_2)\right)\response{.}
 \end{aligned}
 \end{equation}
  For simplicity, we again drop the tilde on $u$. $R_2$ and $P_2$ are analytic functions in the phase space variables $u,v$, depending smoothly on $r_2,\mu_2\sim 0$. 
  
Notice that the eigenvalues of $A(v,\mu_2)$ are $\pm \lambda^{-1} i+\mathcal O(v,\mu_2)$. Consequently, for all $\mu_2$ sufficiently small, we have that 
\begin{align*}
 (q I-A(0,\mu_2))^{-1}
\end{align*}
exists for all $q\in \mathbb R$.  
\begin{lemma}\lemmalab{qIAinv}
 Fix $\delta>0$ small enough. 
 Then there exists a constant $c>0$ such that
\begin{align*}
 \Vert (qI-A(0,\mu_2))^{-1}\Vert \le \frac{c}{\vert q\vert +1},
\end{align*}
for all $q\in \mathbb R$ and all $\mu_2 \in [-\delta,\delta]$. \response{Here $\Vert \cdot \Vert$ denotes the operator-norm induced from the Euclidean norm $\vert \cdot\vert$ on $\mathbb R^2$.}

\end{lemma}
\begin{proof}
Let $\chi(\mu_2)$ and $\overline \chi(\mu_2)$ denote the complex conjugated eigenvalues of $A(0,\mu_2)$. We have 
\begin{align}
\response{\chi(0)=i.} \eqlab{chi0}
\end{align}
The existence of $\tilde c>0$ such that 
\begin{align*}
 \Vert (qI-A(0,\mu_2))^{-1}\Vert \le \tilde c \vert \chi(\mu_2) - q\vert^{-1},
\end{align*}
for all $q\in \mathbb R$ and all $\mu_2\in [-\delta,\delta]$ is elementary, provided $\delta>0$ is small enough. Moreover, using \eqref{chi0} and continuity of $\chi$, we have
\begin{align*}
 \vert \chi(\mu_2)-q\vert>\vert \frac12 i - q\vert>\frac14 (\vert q\vert +1),
\end{align*}
for all $q\in \mathbb R$ and all $\mu_2\in [-\delta,\delta]$, upon restricting $\delta>0$ further if necessary. The result therefore follows, with $c=4\tilde c$.
\end{proof}

  In the following, we will suppress the dependency of $\mu_2\in [-\delta,\delta]$. Specifically, we therefore write $A(v,\mu_2)$ as $A(v)$. 
    Subsequently, we then write \eqref{tildeuv} as a first order system in the following form:
\begin{equation}\eqlab{eqnsm}
\begin{aligned}
 r_2 v \frac{du}{dv} -A(0) u&= \left(A(v)-A(0)\right) u+ r_2 R_2(u,v,r_2)\\
 &- r_2^2 P_2(u,v,r_2)\frac{du}{dv}.
\end{aligned}
\end{equation}
We will prove the existence of the invariant manifold by solving this equation through a fixed-point argument on a closed subset of $\Omega^2(\nu)$.
  
%
%

Now, consider first
\begin{align}
 r_2v\frac{du}{dv} - A(0) u =F(v),\eqlab{lineqn}
\end{align}
with $F:=\sum_{n=0}^{\infty}F_n (\cdot)^n\in \Omega^2(\nu)$. Then we have the following:
\begin{lemma}\lemmalab{Tmap}
Let $0\le r_2\ll 1$ and define the linear operator $T_n:\mathbb R^2\rightarrow \mathbb R^2$ by 
\begin{align*}
 T_n(F_n) = (r_2 \response{n} I - A(0))^{-1} F_n,\quad F_n\in \mathbb R^2.
\end{align*}
Then
\begin{align}
\Vert T_n\Vert \le \frac{c}{r_2n+1},\eqlab{Tkmap}
\end{align}
with $c>0$ from \lemmaref{qIAinv}
and 
\begin{align}
 u(v) =T(F)(v):= \sum_{n=0}^{\infty}T_n(F_n) v^n,\eqlab{Tdef}
\end{align}
is the unique real-analytic solution of \eqref{lineqn} in $\Omega^2(\nu)$. 
The operator $T:\Omega^2(\nu)\rightarrow \Omega^2(\nu)$ defined by \eqref{Tdef} is linear and continuous: 
\begin{align}
 \Vert T(F) \Vert_{\Omega^2} \le c\Vert F\Vert_{\Omega^2},\eqlab{Tcont}
\end{align}
for all $F\in \Omega^2(\nu)$. 

\end{lemma}
\begin{proof}
 Straightforward. \eqref{Tkmap} follows from \lemmaref{qIAinv} which gives \eqref{Tcont} upon using \eqref{Tdef}. 
\end{proof}
The following is also important. 

\begin{lemma}\lemmalab{estTfp}
 Suppose that $P\in \Omega^1(\nu),F \in \Omega^2(\nu)$. Then $T(r_2 P F')\in \Omega^2(\nu)$ and
 \begin{align*}
  \Vert T(r_2 P F') \Vert_{\Omega^2} \le c \nu^{-1} \Vert P\Vert_{\Omega^1} \Vert F\Vert_{\Omega^2},
 \end{align*}
 \responsenew{for all $0\le  r_2\ll 1$}.
\end{lemma}
\begin{proof}
Write 
 $P(v) = \sum_{n=0}^\infty P_{n}v^n$ and $F(v)=\sum_{n=0}^\infty F_n v^n$. Then  $F'=\sum_{n=0}^\infty (n+1) F_{n+1} (\cdot)^n:B_\nu\rightarrow \mathbb C$ is analytic. Moreover, by Cauchy's product rule
\begin{align*}
T\left(r_2 P F'\right)(v) =r_2 \sum_{n=0}^\infty T_n\left(\sum_{l=0}^{n+1} P_{n+1-l} l F_l \right)v^n, 
\end{align*}
such that 
\begin{align}
 \vert T\left(r_2 P F'\right)(v)\vert &\le  \sum_{n=0}^\infty \Vert T_n\Vert r_2(n+1)\sum_{l=0}^{n+1}  \vert P_{n+1-l} \vert \vert F_l\vert\nu^k\nonumber\\
 &\le c \sum_{n=0}^\infty \sum_{l=0}^{n+1}  \vert P_{n+1-l}\vert \vert F_l \vert \nu^n\nonumber\\ 
 &\le c \nu^{-1} \sum_{n=0}^\infty \vert P_{n}\vert \nu^n \sum_{n=0}^\infty \vert F_{n}\vert \nu^n\nonumber\\
 &= c \nu^{-1} \Vert P\Vert_{\Omega^1} \Vert F\Vert_{\Omega^2},\eqlab{estterm3}
\end{align}
for all $v\in B_\nu$ and all $0\le r_2\ll 1$. Here we have used \eqref{Tkmap}.
\end{proof}
Finally, we state the following regarding composition of analytic function with $u\in \Omega^2(\nu)$. 
\begin{lemma}\lemmalab{composition}
Fix $k\in \mathbb N$ and set $$V(u,v)= \sum_{\vert l\vert=0}^\infty \sum_{m=0}^\infty  V_{lm} u^l v^m,$$ where
\begin{align*}
 l:=(l_1,l_2)\in \mathbb N_0,\quad \vert l\vert= l_1+l_2,\quad u^l :=u_1^{l_1}u_2^{l_2},
\end{align*}
and 
$$V_{lm}\in \mathbb R^k,\quad \vert V_{lm}\vert \le K \rho_V^{-\vert l\vert-m},$$ for some $K>0,\rho_V>0$, and all $l\in \mathbb N_0^2,m\in \mathbb N$. 

 Next, consider $$u=\sum_{n=0}^\infty h_n(\cdot)^n \in \Omega^2(\nu),$$ and suppose that 
 \begin{align}
 0<\nu<\rho_V,\quad \Vert u\Vert_{\Omega^2}\le \sigma<\rho_V.\eqlab{ucond}
 \end{align} Then the following holds
%
 \begin{align}
  V\left(\sum_{n=0}^\infty h_n(\cdot)^n,\cdot\right)\in \Omega^k(\nu),
 \end{align}
with
\begin{align*}
 \Vert V\left(\sum_{n=0}^\infty h_n(\cdot)^n,\cdot\right)\Vert_{\Omega^k} \le \frac{K}{(1-\rho_V^{-1}\sigma)^2(1-\rho_V^{-1}\nu)}.
\end{align*}

 \end{lemma}
\begin{proof}
The result follows from a direct calculation: We write $$V\left(\sum_{n=0}^\infty h_n v^n,v \right)=\sum_{n=0}^\infty p_n v^n,$$ by expanding out and using Cauchy's product rule. Let $h_n=(h_n^{(1)},h_n^{(2)})$. Then we find
\begin{align*}
 p_n = \sum_{\vert l\vert=0}^\infty  \sum_{m=0}^t h_m^{(1)}(l_1) \sum_{s=0}^{t-m} h_s^{(2)}(l_2) V_{l(t-m-s)},
\end{align*}
where $h_n^{(i)}(m)\in \mathbb R$, $n\in \mathbb N_0$, $i=1,2$, $m\in \mathbb N_0$, are given by 
\begin{align*}
 \left(\sum_{n=0}^\infty h_n^{(i)} v^n\right)^{m}=:\sum_{n=0}^\infty h_n^{(i)}(m) v^n,
\end{align*}
and Faa di Bruno's formula.
Subsequently, we can then estimate. Using the same approach as in \lemmaref{estTfp} (but going the other way), together with \eqref{ucond}, we find that:
\begin{align*}
 \Vert V\left(\sum_{n=0}^\infty h_n(\cdot)^n,\cdot \right)\Vert_{\Omega^k} &=\sum_{n=0}^\infty \vert p_n\vert \nu^n\\
 &\le \sum_{\vert l\vert=0}^\infty  \sum_{m=0}^\infty  \vert V_{lm}\vert \sigma^{\vert l\vert} \nu^m\\
 &\le K\sum_{\vert l\vert,m=0}^\infty \rho_V^{-\vert l\vert-m} \sigma^{\vert l\vert} \nu^m\\
 &= K \left(\sum_{n=0}^\infty \left(\rho_V^{-1}\sigma\right)^n\right)^2 \sum_{m=0}^\infty \left(\rho_V^{-1}\nu\right)^m\\
 &\le \frac{K}{(1-\rho_V^{-1}\sigma)^2(1-\rho_V^{-1}\nu)},
\end{align*}
as desired.
\end{proof}

To solve \eqref{eqnsm}, we write it as a fixed-point equation:
\begin{align*}
 u = \mathcal L(u),
\end{align*}
where $\mathcal L$ is the nonlinear operator defined by 
\begin{align}
 \mathcal L:\,u\mapsto T\left((A(v)-A(0))u+r_2 R_2(u,v,r_2)-r_2^2 P_2(u,v,r_2) \frac{du}{dv}\right).\eqlab{Leqn}
\end{align}
We will consider the closed subset $$\Omega_1^2(\nu,\sigma)\subset \Omega^2(\nu),$$ defined by $\Vert m\Vert_{\Omega^2} \le \sigma$, $m\in \Omega^2(\nu)$. \response{$\Omega_1^2(\nu,\sigma)$ is -- by the completeness of $\Omega^2(\nu)$ -- a complete metric space.}
\begin{lemma}
 There exist constants $\nu>0$, $\sigma>0$ and $r_{20}>0$ such that $\mathcal L$ is a contraction on $\Omega_1^2(\nu,\sigma)$ for all $0<r_2\le r_{20}$.
\end{lemma}
\begin{proof}
First, we show that $\mathcal L:\Omega_1^2(\nu,\sigma)\rightarrow \Omega_1^2(\nu,\sigma)$ is well-defined for $\nu>0$ and $\sigma>0$ small enough. For this, we use the linearity of $T$ and estimate each of the resulting terms. Notice that $R_2$ and $P_2$ are  real analytic functions, having power series representations with radius of convergence $\rho_{R_2},\rho_{P_2}\ge c>0$ (uniformly in $r_2,\mu_2\sim 0$). Consequently, upon composition with $u\in \Omega_1^2(\nu,\sigma)$ the resulting functions belong to $\Omega^2(\nu)$, $\Omega^1(\nu)$, respectively, for $\nu>0$ and $\sigma>0$ small enough, and are uniformly bounded there, see \lemmaref{composition}. 
We therefore have\response{
\begin{align}
 \Vert \mathcal L(u)\Vert_{\Omega^2} \le c_1 \nu \sigma +c_2 r_2+ c_3 \nu^{-1} \sigma r_2,\eqlab{Lest1}
\end{align}
for $c_1,c_2,c_3>0$ large enough and all $\nu,\sigma,r_2>0$ sufficiently small,}
upon estimating each of the resulting terms. In particular, the last term in the estimate \eqref{Lest1} comes from estimating 
%
%
%
 \begin{align*}
  T\left(r_2^2 P_2(u,v,r_2)\frac{du}{dv}\right),
 \end{align*}
 using \lemmaref{estTfp} and \lemmaref{composition} (for the uniform boundedness of $\Vert P_2\Vert$). 

Next, for the Lipschitz constant of $\mathcal L$, we proceed in an analogously way, writing
\begin{equation}\eqlab{LLips}
\begin{aligned}
 \mathcal L(u_1)-\mathcal L(u_2) &= T((A(v)-A(0))(u_1-u_2)) -T\left(r_2^2 P_2(u_1,v,r_2) \frac{d}{dv}(u_1-u_2)\right)\\
 &+T(r_2 (R_2(u_1,v,r_2)-R_2(u_2,v,r_2))+T\left(r_2^2 \left(P_2(u_2,v,r_2)-P_2(u_1,v,r_2)\right) \frac{du_2}{dv}\right).
\end{aligned}
\end{equation}
The first two terms can easily be estimated in $\Omega^2(\nu)$ by 
\begin{align*}
 c_4 \nu \Vert u_1-u_2\Vert_{\Omega^2},\quad  \frac12 c_5 \nu^{-1} r_2 \Vert u_1-u_2\Vert_{\Omega^2},
\end{align*}
respectively,  for some constants $c_4,c_5>0$ and all $\nu,\sigma,r_2>0$ small enough,
using $A(v)-A(0)=\mathcal O(v)$ and \lemmaref{estTfp} for the estimate of the second term. Now, regarding the estimate of the final two terms of \eqref{LLips}, we write $$V_2(u_1,v,r_2)-V_2(u_2,v,r_2)=\left[\int_0^1 D V_2(u_1(1-s)+u_2s,v,r_2)ds\right](u_1-u_2),$$ for $V=R,P$. By proceeding as in \lemmaref{composition}, using the analyticity of $R_2$, $P_2$, we find that $\left[\cdots\right]$ define bounded linear operators from $\Omega^2(\nu)$ to $\Omega^k(\nu)$, $k=1,2$, respectively, for all $u_1,u_2\in \Omega_1(\nu,\sigma)$ with $\nu,\sigma>0$ sufficiently small. Following \lemmaref{Tmap}, we then obtain an upper bound in $\Omega^2(\nu)$ for the sum of the final two terms of \eqref{LLips}:
\begin{align*}
 \frac12 c_5\nu^{-1} r_2 \Vert u_1-u_2\Vert_{\Omega^2},
\end{align*}
 upon increasing $c_5>0$ further if necessary.
This holds true for all for all $\nu,\sigma,r_2>0$ small enough.
In total, we have
\begin{align*}
  \Vert \mathcal L(u_1)-\mathcal L(u_2)\Vert_{\Omega^2}\le \left(c_4 \nu+c_5 \nu^{-1} r_2\right) \Vert u_1-u_2\Vert_{\Omega^2},
\end{align*}
for all $\nu,\sigma,r_2>0$ sufficiently small. 
This completes the proof.
%
\end{proof}
By Banach's fixed point theorem, we obtain a unique fixed-point of $\mathcal L:\Omega_1^2(\nu,\sigma)\rightarrow \Omega_1^2(\nu,\sigma)$. This fixed point $\tilde u=\tilde m(\cdot,r_2,\mu_2)\in \Omega_1^2(\nu,\sigma)$ gives our desired locally invariant manifold of \eqref{tildeuv}. Transforming the result back to the $(x_2,y_2,z_2)$-variables gives the desired statement of \propref{this} for $k=0$. To obtain the $C^k$-smoothness of $\tilde m(\cdot,r_2,\mu_2)$ with respect to $r_2,\mu_2$, we proceed in the usual way by differentiating \eqref{eqnsm}. This produces variational equations for the partial derivatives of $\tilde m$ and these equations can then be solved successively up to some fixed order $k\in \mathbb N$  using the same approach as for $k=0$. 
We leave out further details.

\end{proof}
The invariant manifold $Z_{2,r_2}$ is a slow manifold extending uniformly with respect to $r_2$ across the degenerate set $y_2=0$, $\mu_2=0$. It is a subset of the stable (unstable) set of \eqref{y2eq} for $\lambda<0$ ($\lambda>0$, respectively).

\subsection{Melnikov theory}\seclab{Melnikov}
We now straighten out the slow manifold of \propref{slowmanifold} by writing: 
\begin{align}
\begin{pmatrix}
    x_2\\
    z_2 
   \end{pmatrix} = m(y_2,r_2,\mu_2) + \tilde u.\eqlab{tildeust}
\end{align}
Then the invariant manifold corresponds to $\tilde u=0$ and consequently:
\begin{equation}\eqlab{tildeuy22}
\begin{aligned}
 \dot{\tilde u} &= B(\tilde u,y_2) \tilde u + r_2 R_2(\tilde u,y_2,r_2,\mu_2) \tilde u,\\
 \dot y_2 &=r_2\left(\frac12 \mu_2 + \lambda y_2+P_2(\tilde u,y_2,r_2,\mu_2)\right),
\end{aligned}
\end{equation}
where $P_2(0,y_2,0,\mu_2)=0$ and 
\begin{align*}
 B(\tilde u,y_2) = \begin{pmatrix}
                           0 & -1\\
                           1 & 2y_2+\tilde u_2
                          \end{pmatrix},
\end{align*}
for $\tilde u=(\tilde u_1,\tilde u_2)$. For $r_2=y_2=0$, we have $m(0,0,\mu_2)=0$ and hence $\tilde u$ reduces to $(x_2,z_2)$ in this case. 

We now again drop the tildes.  The $u$-plane is therefore for $r_2=0$ also filled with the periodic orbits $\phi_h(t)=(x_{2h}(t),z_{2h}(t))$, intersecting the negative $x_2$-axis in $(-h,0)$ with $h>0$. Moreover, due to the invariance of $u=0$, the return map from $u_1<0,u_2=0$ to itself, mapping $(-h,0,y_{20})$ to $(u_1(T),0,y_2(T)$, with $T=T(h,y_2,r_2,\mu_2)$ being the transition time, is well-defined for all $h>0$, $0\le r_2\ll 1$, $\mu_2 \in [-\delta,\delta]$. In fact, polar coordinates in the $u$-plane is well-defined for all $y_2\in [-\nu,\nu]$, $\mu_2 \in [-\delta,\delta]$, and as a result we obtain:
\begin{lemma}\lemmalab{extension}
The return map has a smooth extension to $h=0$ with $T(0,y_2,r_2,\mu_2)=2\pi+\mathcal O(y_2,r_2)$. 
\end{lemma}

We now obtain fixed points of the return map using Melnikov theory to perturb away from the family of period orbits $u=\phi_h(t)$, $h>0$, within $y_2=0$ for $\mu_2=r_2=0$. For this purpose,
let
\begin{align*}
 A_h(t):=\begin{pmatrix}
                           0 & -1\\
                           1 & 2z_{2h}(t)
                          \end{pmatrix},
\end{align*}
denote the linearization around $u=\phi_h(t)$, $y_2=r_2=0$.
 We then write 
\begin{align}
 u &= \phi_h(t)+\tilde u,\quad 
 y_2=\tilde y_2.\eqlab{tildeutildey}
\end{align}
and let $\tilde u(t,h,y_2,r_2,\mu_2), \tilde y_2(t,h,y_2,r_2,\mu_2)$ denote the solutions of the resulting differential equations:
\begin{align*}
\dot{\tilde u} &=A_h(t) \tilde u+\bigg\{B(\phi_h(t)+\tilde u,\tilde y_2)(\phi_h(t)+\tilde u)-A_h(t)\tilde \phi_h(t)- A_h(t)\tilde u\\
&+r_2 R_2(\phi_h(t)+\tilde u,\tilde y_2,r_2,\mu_2)(\phi_h(t)+\tilde u) \bigg\},\\
 \dot{\tilde y}_2&=r_2\left(\frac12 \mu_2 + \lambda \tilde y_2+P_2(\phi_h(t)+\tilde u,\tilde y_2,r_2,\mu_2)\right),
\end{align*}
with initial conditions $\tilde u(0,h,y_2,r_2,\mu_2)=0,\tilde y_2(0,h,y_2,r_2,\mu_2)=y_2$. We have
\begin{align*}
 \tilde u(t,0,y_2,r_2,\mu_2)\equiv 0,
\end{align*}
due to $\phi_0(t)\equiv 0$ and the invariance of $u=0$ for \eqref{tildeuy22}. 
 Let $\Phi_h(t,s)$ denote the state-transition matrix associated with $A_h(t)$. Then by variation of constants, we have that
 \begin{align}
 \tilde u(T,h,y_2,r_2,\mu_2) &= \int_0^{T}\Phi_h(T,t)  \left\{\cdots \right\}dt,
  \eqlab{tildeueqn}\\
 \tilde y_2(T,h,y_2,r_2,\mu_2)&=y_2+\int_0^T r_2\left(\cdots\right)dt,\nonumber
 \end{align}
 where
 \begin{align*}
  \left\{\cdots \right\}&=B(\phi_h(t)+\tilde u(t,h,y_2,r_2,\mu_2),\tilde y_2)(\phi_h+\tilde u(t,h,y_2,r_2,\mu_2))-A_h(t)\tilde \phi_h(t)\\
  &- A_h(t)\tilde u(t,h,y_2,r_2,\mu_2)+r_2 R_2(\phi_h(t)+\tilde u(t,h,y_2,r_2,\mu_2),\tilde y_2(t,h,y_2,r_2,\mu_2),r_2,\mu_2) \\
  &\times(\phi_h(t)+\tilde u(t,h,y_2,r_2,\mu_2)),\\
\left(\cdots\right)&=\frac12 \mu_2 + \lambda \tilde y_2(t,h,y_2,r_2,\mu_2)+P_2(\phi_h(t)+\tilde u(t,h,y_2,r_2,\mu_2),\tilde y_2(t,h,y_2,r_2,\mu_2),r_2,\mu_2).
 \end{align*}
 Recall that $T(h,y_2,r_2,\mu_2)$ denotes the transition time, with $T_0(h):=T(h,0,0,0)$ being the period of the periodic orbits $\phi_h(t)=(x_{2h}(t),z_{2h}(t))$. Since the eigenvalues of $B(0,0)$ are $\pm i$, we have 
\begin{align}
 \phi_h(t) = -h\begin{pmatrix} \cos t+\mathcal O(h)\\ \sin t+\mathcal O(h))\end{pmatrix},\eqlab{phih}
\end{align}
for $h\rightarrow 0^+$. Moreover: 
 \begin{align}
  T_0(0):=\lim_{h\rightarrow 0^+}T_0(h)= 2\pi,\eqlab{T00}
 \end{align}
 see also \lemmaref{extension}.
We then consider the adjoint system:
\begin{align*}
 \dot \psi &= -A_h(t)^T \psi,
\end{align*}
and let $\psi_{h}(t)$ denote the solution with $\psi_h(T)=(1,0)$. Then a simple calculation shows the following:
\begin{lemma}
\begin{align*}
 \psi_{h}(t) &= -h^{-1} e^{\int_t^T 2z_{2h}(s)ds}  \begin{pmatrix}
                  z_{2h}'(t)\\
                  -x_{2h}'(t)
                 \end{pmatrix}.
\end{align*}
Moreover,
 \begin{align}
 \psi_{h}(T) \cdot  \tilde u(T,h,y_2,r_2,\mu_2)  =
 \int_0^{T}\psi_{h}(t) \cdot  \left\{\cdots\right\} dt.\eqlab{phihdotu}
\end{align}
\end{lemma}
\begin{proof}
 This is standard and follows from the theory of adjoint equations, see \cite{matsumoto1993a}.  In particular, for \eqref{phihdotu}, \responsenew{with $\tilde u$ given by \eqref{tildeueqn}}, we use $\psi_{h}(T)^T \Phi_h(T,s) = \psi_{h}(s)^T$. 
\end{proof}

We therefore define the Melnikov functions
\begin{align*}
 \Delta_1(h,y_2,r_2,\mu_2) & = \left[x_{2h}(T)-h\right]+ \int_0^{T}\psi_{h}(t) \cdot  \left\{\cdots\right\} dt,\\
 \Delta_2(h,y_2,r_2,\mu_2) &=\int_0^T \left(\cdots\right)dt.
\end{align*}
so that roots of $\Delta:=(\Delta_1,\Delta_2)$ for $h>0$ (using \eqref{tildeutildey}) correspond to fixed points of the return map and periodic orbits. We have:
\begin{lemma}\lemmalab{Deltah}
$\Delta$ also extends smoothly to $h=0$, with $\Delta_1(0,y_2,r_2,\mu_2)\equiv 0$ and $\Delta(h,0,0,0)\equiv 0$ for all $h\ge 0$.
\end{lemma}
\begin{proof}
 The extension of $\Delta$ follows from the invariance of $u=0$ of \eqref{tildeuy22} and the fact that $T$ extends smoothly to $h=0$, recall \lemmaref{extension}. The invariance of $u=0$ and the fact that $\phi_0(t)\equiv 0$ also gives $\{\cdots\}=0$ for $h=0$ and therefore $\Delta_1(0,y_2,r_2,\mu_2)\equiv 0$. Finally, $\Delta(h,0,0,0)=0$ since $u=(x_2,z_2)=\phi_h(t)$ is a solution of \eqref{layer2} within $y_2=0$ and $r_2=0$.
\end{proof}

Following this lemma, the function
\begin{align*}
 \widehat \Delta:=(h^{-1} \Delta_1,\Delta_2),
\end{align*}
is well-defined and smooth on $h\ge 0$. 
Moreover:
\begin{lemma}
$\widehat \Delta(h,0,0,0)\equiv 0$ and
\begin{align}
 \widehat \Delta_{y_2}'(h,0,0,0)&=\begin{pmatrix} -2 \int_0^{T_0(h)} e^{\int_t^T 2z_{2h}(s)ds} (h^{-1} z_{2h})(t)^2                                             
dt\\ *\end{pmatrix},\eqlab{Deltay2}\\
\widehat \Delta_{\mu_2}'(h,0,0,0) &= \begin{pmatrix}
                            0 \\
                            \frac12 T_0(h)
                           \end{pmatrix},\eqlab{Deltamu}
\end{align}
for all $h\ge 0$,
with $*$ denoting a quantity that is not important. 
\end{lemma}
\begin{proof}
 $\widehat \Delta(h,0,0,0)\equiv 0$ follows from \lemmaref{Deltah} and \eqref{Deltamu} is obvious. To show \eqref{Deltay2}, we compute for $h>0$:
 \begin{align*}
  \frac{\partial \Delta_1}{\partial y_2}(h,0,0,0)&=\int_0^{T_0} \psi_h(t) \cdot \begin{pmatrix}
                                                                                0 & 0 \\
                                                                                0 & 2
                                                                               \end{pmatrix}\phi_h(t) dt\\
                                                                               &=  2h^{-1}  \int_0^T e^{\int_t^T 2z_{2h}(s)ds} x_{2h}'(t) z_{2h}(t)dt  \\
                                                                               &=-2h^{-1} \int_0^T e^{\int_t^T 2z_{2h}(s)ds} z_{2h}(t)^2dt.
 \end{align*}
 $h^{-1}z_{2h}$ extends to all $h\ge 0$ by \eqref{phih},
 which then completes the proof.
\end{proof}

The Jacobian matrix
\begin{align*}
 \widehat \Delta_{(y_2,\mu_2)}(h,0,0,0): = \begin{pmatrix}
                          \widehat\Delta_{y_2}'(h,0,0,0) & \widehat \Delta_{\mu_2}'(h,0,0,0)
                        \end{pmatrix},
\end{align*}
is therefore regular for all $h>0$, regardless of what $*$ is, also for $h=0$ due to \eqref{phih} and \eqref{T00}:
\begin{align*}
\widehat \Delta_{y_2}'(\textnormal{\textbf 0}) &= \begin{pmatrix}  -2\pi \\ *\end{pmatrix},\quad 
\widehat \Delta_{\mu_2}'(\textnormal{\textbf 0}) = \begin{pmatrix}
                            0 \\
                            \pi
                           \end{pmatrix}.
                           \end{align*}
                           \begin{proposition}\proplab{smallcycles}
                            Fix any $h_0>0$. Then there exist smooth functions $\overline y_2(h,r_2),\overline \mu_2(h,r_2)$ with $\overline y_2(h,0)=0$, $\overline \mu_2(h,0)=0$ such that $\widehat \Delta(h,\overline y_2(h,r_2),r_2,\overline \mu_2(h,r_2))=0$ for all $h\in [0,h_0]$ and all $0\le r_2\ll 1$.
                           \end{proposition}
                           \begin{proof}
                            Follows \responsenew{immediately} from the implicit function theorem.
                           \end{proof}
                                                    \begin{corollary}
                           Let $\mu_{H}(r_2):=r_2 \overline \mu_2(0,r_2)$. Then \eqref{eqn2} undergoes a Hopf bifurcation at $\mu=\mu_{H}(r_2)$ for $0<r_2\ll 1$. 
                           \end{corollary}
                           \begin{proof}
%
The statement follows directly from \propref{smallcycles}.
                           \end{proof}
\begin{remark}\remlab{stability}
                           We do not aim to describe the stability of the cycles. This is more difficult. However, sufficiently close to the Hopf bifurcation, one can describe stability in terms of the signs of $\lambda$ and $\frac{\partial^2 \overline \mu_2}{\partial h^2}(0,r_2)$; here the latter quantity relates directly to the first Lyapunov coefficient of the Hopf (within a center manifold). 
                           In particular, if $\lambda>0$ then $\mu_2\gtrless \mu_H(r_2)$ means that the equilibrium is on the attracting/repelling side of $Z_{2,r_2}$, respectively. Therefore sufficiently close to the equilibrium, the limit (Hopf) cycles are unstable/stable if $\frac{\partial^2 \overline \mu_2}{\partial h^2}(0,r_2)\gtrless 0$ for $\lambda>0$.                      
                             If $\lambda<0$, then the inequalities for $\mu_2$ regarding the ``normal stability'' of \eqref{y2eq} become $\mu_2 \lessgtr \mu_H(r_2)$ and the Hopf cycles are unstable/stable if $\frac{\partial^2 \overline \mu_2}{\partial h^2}(0,r_2)\lessgtr 0$ in this case.  
                            It is possible to compute a leading order expression for $\frac{\partial^2 \overline \mu_2}{\partial h^2}(0,r_2)$ but we chose not to include the complicated expression in the present paper.  
                            
                            Obviously, one could just compute the Lyapunov coefficient by applying the center manifold reduction. 
                            But -- due to the zero-Hopf bifurcation for $r_2=0$ --  this approach does not guarantee that this quantity describe the stability of limit cycles in a uniform neighborhood (i.e. one that does not shrink in the $(x_2,y_2,z_2)$-space as $\epsilon\rightarrow 0$). Our Melnikov approach using \propref{this} in the analytic setting does imply such uniformity.
                            \end{remark}

By  working in the chart $\bar x=-1$ in the following section, we are able to extend the \responsenew{small cycles, due to \propref{smallcycles}} to ``intermediate'' cycles that include $\mathcal O(1)$ canards.

\section{Existence of intermediate periodic orbits}\seclab{inter}
Our strategy for extending the small periodic orbits to intermediate ones, that connect to cycles of size $\mathcal O(1)$, follows the approach for proving canard cycles in $\mathbb R^2$: Working in the entry chart $\bar x=-1$, we fix a section along $z_1=0$, in a neighborhood of $(r_1,y_1,\epsilon_1)=0$ with $\mu\sim 0$, and flow the points forward and backward and then measure their separation on a section $\{z_2=0\}$ in the scaling chart $\bar \epsilon=1$ transverse to $\gamma_{2,0}(0)$, recall \eqref{gamma20}. Based upon expansions from the solution of a Shilnikov problem \cite{deng1988a}, we define appropriate scalings that allow us to solve for roots of the separation function by applying the implicit function theorem. 
\subsection{Analysis in the $\bar x=-1$-chart}
We first consider the $\bar x=-1$-chart. Inserting \eqref{barxneg1} into (\eqref{fast},$\dot \epsilon=0$) gives
\begin{equation}\eqlab{barxneg1eqns}
\begin{aligned}
\dot \epsilon_1&=\epsilon_1^2 \left[y_1-(\mu+1)z_1+r_1F_1(\epsilon_1,r_1,y_1,z_1,\mu)\right],\\
 \dot r_1 &=-\frac12 r_1 \epsilon_1 \left[y_1-(\mu+1)z_1+r_1F_1(\epsilon_1,r_1,y_1,z_1,\mu)\right],\\
  \dot y_1&=\epsilon_1 \left(\frac12 \mu+r_1 G_1(\epsilon_1,r_1,y_1,z_1,\mu)\right)+\frac12 \epsilon_1 y_1\left[y_1-(\mu+1)z_1+r_1F_1(\epsilon_1,r_1,y_1,z_1,\mu)\right] ,\\
 \dot z_1&=-1+z_1^2+r_1 z_1 H_1(\epsilon_1,r_1,y_1,z_1,\mu)+\frac12 \epsilon_1 z_1 \left[y_1-(\mu+1)z_1+r_1F_1(\epsilon_1,r_1,y_1,z_1,\mu)\right], 
\end{aligned}
\end{equation}
where $ F_1 = r_1^{-2} F$, $G_1=r_1^{-1} G$, and $H_1 = r_1^{-2} H$, each being smooth, also upon extending to $r_1=0$. The points $(0,0,y_1,- 1)$ are partially attracting, whereas $(0,0,y_1,1)$ are partially repelling. In particular, in each case the linearization has a single nonzero eigenvalue (positive/negative, respectively). Consequently, by center manifold theory we have the following.

\begin{lemma}\lemmalab{centermanifold}
Fix any $k\in \mathbb N$. Then there exists a constant $\sigma>0$ sufficiently small such that the following holds: There exist two three-dimensional center manifolds $M_{1,a}$ and $M_{1,r}$ of the sets $(0,0,y_1,- 1)$ and $(0,0,y_1,1)$, $y_1\in [-\sigma,\sigma]$, having the following graph representations:
 \begin{align*}
  M_{1,a}:\quad z_1 &= m_{1,a}(\epsilon_1,r_1,y_1,\mu),\\
  M_{1,r}:\quad z_1 &= m_{1,r}(\epsilon_1,r_1,y_1,\mu),
 \end{align*}
for $0\le r_1,\epsilon_1,\vert y_1\vert ,\vert \mu\vert\le \sigma$. The functions $m_{1,a}$ and $m_{1,r}$ are $C^k$-smooth functions on the specified domain and satisfy $m_{1,a}(0,0,y_1,\mu)\equiv -1,m_{1,r}(0,0,y_1,\mu)\equiv 1$ as well as
\begin{align}
 m_{1,a}\left(\epsilon_1,0,-\mu\sqrt{1+\frac{\epsilon_1}{2}},\mu\right)\equiv -\sqrt{1+\frac{\epsilon_1}{2}},\quad m_{1,r}\left(\epsilon_1,0,\mu\sqrt{1+\frac{\epsilon_1}{2}} ,\mu\right)\equiv \sqrt{1+\frac{\epsilon_1}{2}}.\eqlab{m1ar}
\end{align}
 
\end{lemma}
%
The properties in \eqref{m1ar} are consequences of the invariant sets $\gamma_{2,0}(\mu)$ in the $\bar \epsilon=1$-chart, recall \eqref{gamma2}, which in the present chart become $r_1=0,y_1=\mp \mu \sqrt{1+\frac{\epsilon_1}{2}},z_1=\mp \sqrt{1+\frac{\epsilon_1}{2}}$ and $\epsilon_1\ge 0$; we denote these sets by $\gamma_{1,0,a}(\mu)$ and $\gamma_{1,0,r}(\mu)$, respectively. We illustrate the dynamics in the $\bar x=-1$-chart in \figref{entry1}.

\begin{figure}[h!]
\begin{center}
\subfigure[$\mu=0$]{\includegraphics[width=.25\textwidth]{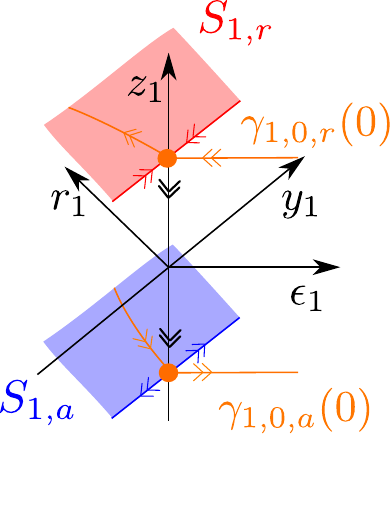}}
\subfigure[$\mu=\overline \mu_0(r_1)$]{\includegraphics[width=.25\textwidth]{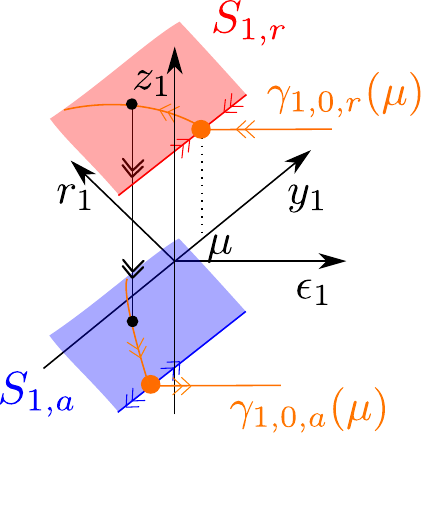}}
\end{center}
\caption{Dynamics in the entry chart $\bar x=-1$. In (a): for $\mu=0$. In (b): for $\mu=\mu_0(r_1,0)$. $\mu_0$ is defined later, but essentially $\mu_0(\tilde r_1,0)$ with $\tilde r_1>0$ is so that for $r_1=\tilde r_1$, $\mu=\mu(\tilde r_1,0)$ there is fast jump connecting $\gamma_0$ to itself.}
\figlab{entry1}
\end{figure}

As is standard, the center manifolds $M_{1,a}$, $M_{1,r}$ provide extensions of the family of Fenichel slow manifolds $\{(S_{a,\epsilon},\epsilon):\epsilon\in [0,\epsilon_0)\}$ as (\eqref{fast},$\dot \epsilon=0$) as foliations (through $r_1^2\epsilon_1=\epsilon=\text{const}.$) of the center manifolds \cite{dumortier_1996}. 

On $M_{1,a}$ and $M_{1,r}$ we have the following desingularized reduced problems
\begin{equation}\eqlab{redMa10}
\begin{aligned}
 \epsilon_1'& = \epsilon_1,\\
 r_1'&=-\frac12 r_1,\\
 y_1'&=\frac{\frac12 \mu+\mathcal O(r_1)}{y_1+(\mu+1)+\mathcal O(\epsilon_1,r_1)}+\frac12 y_1,
 \end{aligned}
 \end{equation}
%
and 
\begin{equation}\eqlab{redMr10}
\begin{aligned}
\epsilon_1'& =- \epsilon_1,\\
  r_1'&=\frac12 r_1,\\
  y_1'&=\frac{\frac12 \mu+\mathcal O(r_1)}{-y_1+(\mu+1)+\mathcal O(\epsilon_1,r_1)}-\frac12 y_1,
 \end{aligned}
\end{equation}
obtained by dividing the right hand sides by $\pm \epsilon_1 \left[y_1-(\mu+1)z_1+r_1F_1(\epsilon_1,r_1,y_1,z_1,\mu)\right]$ and then substituting $z_1=m_{1,a}(\cdots)$ and $z_1=m_{1,r}(\cdots)$, respectively. Notice that the square brackets  $\pm \left[y_1-(\mu+1)z_1+r_1F_1(\epsilon_1,r_1,y_1,z_1,\mu)\right]$ are positive near $M_{1,a}$ and $M_{1,r}$, respectively, and for $\epsilon_1>0$ the divisions therefore correspond to time transformations.

The system \eqref{redMa10} has a hyperbolic equilibrium at $(0,0,-\mu)$ with eigenvalues $-1/2,1,\lambda(\mu)$ where
\begin{align}
 \lambda(\mu) = \frac12 (1-\mu).\eqlab{lambdamu}
\end{align}
Similarly, \eqref{redMr10} has a hyperbolic equilibrium at $(0,0,\mu)$ with eigenvalues $\frac12,-1,-\lambda(\mu)$. There are therefore strong resonances of both systems when $\mu=0$.
\responsenew{
\subsubsection{Proof of \lemmaref{lemma0}}\seclab{lemma0}
We can complete the proof of \lemmaref{lemma0} and the existence of the singular cycle $\Gamma_{0,h}$ by working in the chart $\bar x=-1$. For $\epsilon_1=0$, $M_{1,a}$ and $M_{1,r}$ become $S_{1,a}$ and $S_{1,r}$, respectively. 
Moreover, $\gamma_0(\mu)\cap S_{a}$ and $\gamma_0(\mu)\cap S_r$ become stable and unstable manifolds of $(0,-\mu,-1)$ and $(0,\mu,1)$ for \eqref{redMa10}  and \eqref{redMr10}, respectively, contained within $\epsilon_1=0$. Due to the invariance of the transformation \eqref{barxneg1} with respect to the involution $(r_1,y_1,z_1)=(-r_1,-y_1,-z_1)$, we can write these stable and unstable manifolds as smooth graphs
\begin{align*}
 y_1=\mp Y(\pm r_1,\mu),
\end{align*}
respectively, over $r_1,\mu\sim 0$. Here $Y(0,\mu)\equiv \mu$. We therefore have a singular canard cycle $\Gamma_{0,h}$ intersecting $z=0$ in $(x,y,z)=(-h^2,*,0)$ if and only if 
\begin{align}
 K_0(h,\mu):=-Y(h,\mu)-Y(-h,\mu)=0.\eqlab{L0func}
\end{align}
(Here $*$ is a quantity that is not important.)
We have $K_0(0,0)=0$ and $\frac{\partial }{\partial \mu}K_0(0,0)=-2$.
We can therefore solve \eqref{L0func} for $\mu=\overline \mu_0(h)$ by the implicit function theorem.
\qed}

In the following, we now describe two transition mappings $P_{1,a}$ and $P_{1,r}$ from $\{z_1=0\}$ to $\{\epsilon_1=\epsilon_{11}\}$, \responsenew{$\epsilon_{11}>0$ fixed and small enough}, obtained by applying the forward and backward flow, respectively. See illustration in \figref{entry1}. We will describe each mapping in appropriate ``normal form coordinates''. Since the two mappings are similar, we focus on $P_{1,a}$.
\begin{lemma}\lemmalab{transformnearM1a}
 Consider \eqref{barxneg1eqns} near $z_1=-1$ and fix any $k\in \mathbb N$. Then for $\sigma>0$ small enough, there exists a regular transformation of time and a $C^{k+3}$-smooth diffeomorphism $(y_1,z_1)\mapsto (y_{1,a},z_{1,a})$, $y_1,z_1+1\in [-\sigma,\sigma]$, depending $C^{k+3}$-smoothly on $(\epsilon_1,r_1,\mu)$ for $0\le r_1,\epsilon_1,\vert \mu\vert \le \sigma$, such that 
 \begin{equation}\eqlab{y1az1a}
 \begin{aligned}
  y_{1,a} &=h_{1,a}^0(\epsilon_1,r_1,y_1,\mu)+\epsilon_1h_{1,a}^1(\epsilon_1,r_1,y_1,z_1,\mu),\\
  z_{1,a}&= z_1-m_{1,a}(\epsilon_1,r_1,y_1,\mu),
 \end{aligned}
 \end{equation}
 where
 \begin{align}\eqlab{h1a0prop}
 h_{1,a}^0(\epsilon_1,0,-\mu,\mu)\equiv 0, \quad \frac{\partial h_{1,a}^0}{\partial y_1} (\textnormal{\textbf{0}})=1,\quad \frac{\partial h_{1,a}^0}{\partial \mu} (\textnormal{\textbf{0}})=1,\end{align} and \eqref{barxneg1eqns} becomes
 \begin{equation}\eqlab{eqnsredMa1new2}
\begin{aligned}
 \dot \epsilon_1 &=\epsilon_1^2,\\
 \dot r_1 &=-\frac12 r_1 \epsilon_1,\\
 \dot{y}_{1,a} &=\left((\lambda(\mu) +r_1 L_0(r_1,y_{1,a},\mu) )y_{1,a}+r_1 \epsilon_1 L_1(\epsilon_1,r_1,y_{1,a},\mu)\right)\epsilon_1,\\
 \dot{z}_{1,a} &= (-2+L_2(r_1,y_{1,a},z_{1,a},\epsilon_1,\mu))z_{1,a},
 \end{aligned}
\end{equation}
with $L_0$, $L_1$, $L_2$ each $C^k$-smooth and $L_2(\textnormal{\textbf{0}})=0$.
 
\end{lemma}
\begin{proof}
Let $l$ be the value of $k\in \mathbb N$ in \lemmaref{centermanifold}. We will fix $l\gg 1$ in the following.
We then divide the right hand side by 
$$y_1-(\mu+1)z_1+r_1F_1(\epsilon_1,r_1,y_1,z_1,\mu),$$
which is positive near $M_{1,a}$. This defines our regular transformation of time.  
Consequently, 
\begin{equation}\nonumber
\begin{aligned}
\dot \epsilon_1&=\epsilon_1^2,\\
 \dot r_1 &=-\frac12 r_1 \epsilon_1,\\
\dot y_1&=\frac{\epsilon_1 \left(\frac12 \mu+r_1 G_1(\epsilon_1,r_1,y_1,z_1,\mu)\right)}{ y_1-(\mu+1)z_1+r_1F_1(\epsilon_1,r_1,y_1,z_1,\mu)}+\frac12 \epsilon_1 y_1 ,\\
 \dot z_1&=\frac{-1+z_1^2+r_1 z_1 H_1(\epsilon_1,r_1,y_1,z_1,\mu)}{ y_1-(\mu+1)z_1+r_1F_1(\epsilon_1,r_1,y_1,z_1,\mu)}+\frac12 \epsilon_1 z_1 ,\end{aligned}
\end{equation}
We then rectify $M_{1,a}$ to $z_{1,a}=0$ by a $C^{l}$-smooth transformation $z_1=m_{1,a}(\epsilon_1,r_1,y_1,\mu)+z_{1,a}$ and subsequently straighten out the associated stable fibers of $M_{1,a}$ by a $C^{l-3}$-transformation $(\epsilon_1,r_1,y_1,z_{1,a},\mu)\mapsto \tilde y_1$ such that $(\epsilon_1,r_1,\tilde y_1)'$ becomes independent on $z_{1,a}$, see \cite[Lemma 2.2]{deng1990a}.  Then 
\begin{equation}\nonumber
\begin{aligned}
\dot \epsilon_1&=\epsilon_1^2,\\
 \dot r_1 &=-\frac12 r_1 \epsilon_1,\\
  \dot{\tilde y}_1&= \epsilon_1\left(\frac{\frac12 \mu+\mathcal O(r_1)}{\tilde y_1+(\mu+1)+\mathcal O(\epsilon_1,r_1)}+\frac12 \tilde y_1\right),\\
 \dot z_{1,a}&=(-2+\mathcal O(\epsilon_1,r_1,\tilde y_{1},z_{1,a},\mu))z_{1,a}.
\end{aligned}
\end{equation}
By construction $r_1,\epsilon_1,\tilde y_1$ decouples. For $\epsilon_1=0$ it easy to see that $\tilde y_1=y_1$. This shows the form of \eqref{y1az1a} with the $z_1$-dependency of $y_{1,a}$  only entering the $\mathcal O(\epsilon_1)$-term. Indeed, in the remainder of the proof we only transform $\tilde y_1$ in order to normalize the reduced problem on $z_{1,a}=0$:
\begin{equation}\eqlab{redMa1}
\begin{aligned}
\epsilon_1'& = \epsilon_1,\\
r_1' &=-\frac12 r_1,\\
y_1'&=\frac{\frac12 \mu+\mathcal O(r_1)}{y_1+(\mu+1)+\mathcal O(\epsilon_1,r_1)}+\frac12 y_1,
\end{aligned}
\end{equation}
after dividing the right hand side by $\epsilon_1$. We have  here dropped the tilde on $y_1$ and the system then coincides (by construction) with \eqref{redMa10}. 

 Put $y_1=-\mu + \tilde y_1$.
Then
\begin{equation}\eqlab{M1aRedFull}
\begin{aligned}
 \epsilon_1'&=\epsilon_1,\\
 r_1' &=-\frac12 r_1,\\
 y_1' &=\lambda(\mu) y_1+L_1(\epsilon_1,r_1,y_1,\mu),
\end{aligned}
\end{equation}
after dropping the tildes,
with $L_1(\textbf{0},\mu)=\frac{\partial }{\partial y_1}L_1(\textbf 0,\mu)=0$.  First consider the $r_1=0$ subsystem:
\begin{equation}\eqlab{r1Eq0subsystem}
\begin{aligned}
 \epsilon_1'&=\epsilon_1,\\
  y_1' &=\lambda(\mu) y_1+L_1(\epsilon_1,0,y_1,\mu),
\end{aligned}
\end{equation}
having eigenvalues $1$, $\lambda(\mu)=\frac12(1-\mu)$. Due to the invariance of $\epsilon_1=0$, there are no resonant monomials for all $\mu\sim 0$ and
\begin{align*}
 \epsilon_1' &= \epsilon_1,\\
 y_1'&=\lambda(\mu)y_1,
\end{align*}
is a formal normal form. Following \cite{ilyashenko1991a} and \cite[Theorem 2.15 and section 2.7]{dumortier2006a}, the system \eqref{r1Eq0subsystem} can therefore be linearized by a $C^m$-transformation of the form $(\epsilon_1,y_1,\mu) \mapsto \tilde y_1=h_{1,a}(\epsilon_1,y_1,\mu)$ with $h_{1,a}(0,0,\mu)= 0$, $\frac{\partial h_{1,a}}{\partial y_1}(0,0,\mu)=1$ for any $m\in \mathbb N$ (upon increasing $l$ if necessary). We take $m$ large enough. Applying this transformation to \eqref{M1aRedFull}, leaving $r_1$ untouched, it follows that we can take $L_1$ to be of the form $L_1=r_1\widetilde L_1$. Then within $\epsilon_1=0$, we have a saddle \responsenew{(due to $\gamma_0(\mu)$)} for $\mu\sim 0$ and we can straighten out the unstable manifold by an $r_1$-dependent transformation of $y_1$: $(r_1,y_1)\mapsto y_{1,a}$. This gives $\widetilde L=y_{1,a}\widetilde L_0 +\epsilon_1 \widetilde L_1$:
\begin{equation}\eqlab{eqnsredMa1new}
\begin{aligned}
 \epsilon_1' &=\epsilon_1,\\
 r_1' &=-\frac12 r_1,\\
 y_{1,a}' &=(\lambda(\mu) +r_1 L_0(r_1,y_{1,a},\mu) )y_{1,a}+r_1 \epsilon_1 L_1(\epsilon_1,r_1,y_{1,a},\mu),
\end{aligned}
\end{equation}
upon dropping the tildes. Upon composition, the desired coordinate transformation can be written as in \eqref{y1az1a}. With $k$ fixed, we can, upon taking $m$ and $l$ large enough, ensure that \eqref{y1az1a} is $C^{k+3}$. Applying this transformation to \eqref{barxneg1eqns} gives that $L_0$, $L_1$ and $L_2$, upon expanding, are all $C^k$.
 \end{proof}



Since the flow is regular on sets that are uniformly bounded away from $M_{1,a}$ and $M_{1,r}$, we can extend the transformation leading to \eqref{eqnsredMa1new2} to $z_1 \in [-1-\sigma,0]$, say, by the flow-box theorem. We therefore describe $P_{1,a}$ using the coordinates of \eqref{eqnsredMa1new2} from $\{z_{1,a}=-\widetilde m_{1,a}(\epsilon_1,r_1,y_{1,a},\mu)\}$ to $\{\epsilon_1=\epsilon_{11}\}$. Here $\widetilde m_{1,a}(\epsilon_1,r_1,y_{1,a},\mu) :=m_{1,a}(r_1,y_1(\epsilon_1,r_1,y_{1,a},\mu),\epsilon_1,\mu)$ with the smooth function $y_1(\epsilon_1,r_1,y_{1,a},\mu)$ being obtained (by the implicit function theorem) from \eqref{y1az1a} with $z_1=0$:
\begin{align*}
 y_{1,a}=h_{1,a}^0(\epsilon_1,r_1,y_1,\mu)+\epsilon_1 h_{1,a}^1(\epsilon_1,r_1,y_1,0,\mu).
\end{align*}
For simplicity, we denote the transformed mapping $$P_{1,a}(\epsilon_1,r_1,y_{1,a},\mu)=\begin{pmatrix} 
                                                                                                                                                        \epsilon_{11}\\ \sqrt{\epsilon_1 \epsilon_{11}^{-1}}r_1\\
                                                                                                                                                        *\\
                                                                                                                                                        *
                                                                                                                                                                                                                                                                                                                \end{pmatrix}
,$$ by the same symbol. Notice that the $r_1$- and $\epsilon_1$-components of the mapping are just consequences of the conservation of $\epsilon=r_1^2 \epsilon_1$. 

We describe $P_{1,a}$ in \propref{P1a}, but in preparation we first state a result on the \textit{Shilnikov problem}
associated with the $(\epsilon_1,r_1,y_{1,a})$-subsystem, written in the desingularized form \eqref{eqnsredMa1new}. For this, it will be convinient to work with functions $(u,v)\in U\times  V \mapsto f(u,v)\in \mathbb R^l$, $U\subset \mathbb R^n,V\subset \mathbb R^m$, $l,n,m\in \mathbb N$, that are $C^{k}$-smooth, $k\in \mathbb N$, with respect to the $v\in V$, depending continuously on $u\in U$. We write the function space of such functions by $C(u^0v^k)$ and let
\begin{align*}
 \Vert f\Vert_{u^0v^k}:=\sup_{(u,v)\in U\times V,\,0\le \vert i\vert\le k}\vert D^i f(u,v)\vert,
\end{align*}
with $D^i$ being the partial derivative with respect to $v$ of order $i=(i_1,\ldots,i_m)\in \mathbb N_0^m$,
denote the associated Banach norm. Notice that by $C(u^0v^k)$ and $\Vert \cdot\Vert_{u^0v^k}$ we suppress $l$, $n$, $m$, $U$ and $V$; it should be clear from the context what these are. On the other hand, we use $$\Vert f \Vert_{C^k}:=\sup_{u\in U,\,0\le \vert i\vert\le k} \vert D^i f(u)\vert, $$ to denote the usual $C^k$-norm of a $C^k$-smooth function $f:U\rightarrow \mathbb R^l$. 
\begin{proposition}\proplab{this}
Fix $k\in \mathbb N$, $\alpha\in (0,\frac12)$ and consider \eqref{eqnsredMa1new} with the right hand side being $C^{k+1}$. Then there exist constants $\tau_0>0$, $c_1>0$ and $c_2>0$ such that the following holds for all $\delta>0$ small enough: There exists a $C^k$-smooth function $\underline y_{1,a}(t,\tau,\epsilon_{11},r_{10},y_{11},\mu)$, defined on the set given by: $(t,\tau): 0\le t\le \tau, \tau>\tau_0$ and $0\le \epsilon_{11},r_{10},\vert y_{11}\vert,\vert \mu\vert \le c_1 \delta$, such that $y_{1,a}(t)=\underline y_{1,a}(t,\tau,\epsilon_{11},r_{10},y_{11},\mu)$ solves the \textnormal{Shilnikov problem} defined by
 \begin{align*}
 \epsilon_1(\tau)=\epsilon_{11},\quad r_1(0)=r_{10},,\quad y_{1,a}(\tau)=y_{11}.
 \end{align*}
 Moreover, $\underline y_{1,a}$ has the following expansion:
 \begin{align*}
 \underline y_{1,a}(t,\tau,\epsilon_{11},r_{10},y_{11},\mu) = e^{\lambda(\mu)(t-\tau)} (y_{11}+\phi(t,\tau,\epsilon_{11},r_{10},y_{11},\mu)),
 \end{align*}
where $\Vert \phi\Vert_{C^k}\le c_2$; specifically
\begin{align}
 \Vert \phi\Vert_{(t,\tau,\epsilon_{11},r_{10})^0 (y_{11},\mu)^k} \le \delta,\quad \phi(t,\tau,\epsilon_{11},0,0,\mu)\equiv 0. \eqlab{bphi}
\end{align}
Finally, define 
\begin{align*}
 \phi_{\infty}(t,\epsilon_{11},r_{10},y_{11},\mu):=\left(e^{-\int_{t}^\infty  r_1(s) L_0(r_1(s),0,\mu)ds}-1\right)y_{11} ,
\end{align*}
for all $t\ge 0$ with
$r_1(t) = e^{-\frac12 t} r_{10}$.
Then $\phi_{\infty}\in C^k$ and
\begin{align}
 \Vert \phi(t,\tau,\cdot) - \phi_{\infty}(t,\cdot)\Vert_{(\epsilon_{11},r_{10})^0(y_{11},\mu)^k} \le c_2 \delta e^{-\alpha \tau}, \eqlab{phiconv}
\end{align}
for all $\tau>\tau_0$, $0\le t\le \tau$. 

\end{proposition}
\begin{proof}
The proof follows \cite{deng1988a} and we therefore delay the details to \appref{proofprop32}.
\end{proof}


 \begin{proposition}\proplab{P1a}
  Fix $k\in \mathbb N$, any $\delta>0$ and consider \eqref{eqnsredMa1new2} with the right hand side being $C^{k+1}$. Then there exist constants $c>0$, $\chi>0$ and $0<\epsilon_{10}\ll \epsilon_{11}$, $r_{10}>0$ such that the following holds: The transition map $P_{1,a}(\cdot,\mu)$ of \eqref{eqnsredMa1new2} is well-defined on a region $V_{1,a}(\chi)$ defined by
 \begin{align}
    \vert y_{1,a}\vert \le \chi (\epsilon_1 \epsilon_{11}^{-1})^{\lambda(\mu)},\epsilon_1 \in [0,\epsilon_{10}],r_1\in [0,r_{10}],\eqlab{V1a}
 \end{align}
 and on this region, the mapping takes the following form
   \begin{align*}
  P_{1,a}(\epsilon_1,r_1,y_{1,a},\mu) =\begin{pmatrix}
  \epsilon_{11}\\
                                    \sqrt{\epsilon_1 \epsilon_{11}^{-1}}r_1\\
                                    (\epsilon_{1}^{-1} \epsilon_{11})^{\lambda(\mu)} y_{1,a}+\psi_{1,a} (\epsilon_1,r_{1},(\epsilon_{1}^{-1} \epsilon_{11})^{\lambda(\mu)} y_{1,a},\mu)\\
                                    \mathcal O(e^{-c/\epsilon_1})                                    
                                   \end{pmatrix}.
 \end{align*}
 Here $\psi_{1,a}$ is a $C({\epsilon_1}^0(r_1,u,\mu)^k)$-function on the set $\epsilon_1\in [0,\epsilon_{10}],r_1\in [0,r_{10}],0\le \vert u\vert,\vert \mu\vert \le \chi$ with 
 \begin{align}
  \Vert \psi_{1,a}\Vert_{(\epsilon_1,r_1)^0(u,\mu)^k}\le \delta,\quad \psi_{1,a}(\epsilon_1,0,0,\mu)\equiv 0.\eqlab{psi1aprop}
 \end{align}
The $\mathcal O(e^{-c/\epsilon_1})$ remainder term in the $z_{1,a}$-component is also a $C({\epsilon_1}^0(r_1,y_{1,a},\mu)^k)$-function with the order being unchanged upon differentiation up to order $k$ with respect to $(r_1,y_{1,a},\mu)$ for $(\epsilon_1,r_1,y_{1,a})\in V_{1,a}(\chi)$.
 \end{proposition}

\begin{proof}
Since $(\epsilon_1,r_1,y_{1,a})$-decouples, we use \propref{this} with $\tau$ defined by $\epsilon_{11} = e^{\tau}\epsilon_{1}$. Let
 \begin{align}\eqlab{y0eqn}
  y_{1,a} =\underline y_{1,a}(0,\tau,\epsilon_{11},r_{1},y_{11},\mu)=e^{-\lambda(\mu) \tau}[y_{11}+\phi(0,\tau,\epsilon_{11},r_{1},y_{11},\mu)].
 \end{align}
 The $y_{1,a}$-component of the transition map is then defined implicitly by this equation as $y_{1,a}\mapsto y_{11}$.
 Consider therefore $F(u_0,\tau,y_{11},\mu)=0$ with
\begin{align*}
F(u_0,\tau,r_1,y_{11},\mu) := u_0-(y_{11}+\phi(0,\tau,\epsilon_{11},r_1,y_{11},\mu)),
\end{align*}
for $\epsilon_{11}>0$ fixed,
defined by setting the square bracket in \eqref{y0eqn} equal to $u_0$. Then by \propref{this} and the implicit function theorem, we obtain a locally unique solution 
\begin{align}\eqlab{y11eqn}
y_{11}=u_0+\widetilde \psi_{1,a} (\tau,r_{1},u_0,\mu),
\end{align}
of $F=0$ with $\widetilde \psi_{1,a}\in C^k$ defined on $\tau>\tau_0$, $0\le \vert u\vert,r_1,\epsilon_{11},\vert \mu\vert \le \chi$ with $\tau_0^{-1}>0$ and $\chi>0$ small enough. With $\epsilon_{11}>0$ fixed in the folllowing we suppress its dependency from $\widetilde \psi_{1,a}$. Inserting $u_0=e^{\lambda(\mu) \tau}y_{1,a}=(\epsilon_{11} \epsilon_1^{-1})^{\lambda(\mu)}y_{1,a}$ into the right hand side of \eqref{y11eqn} and defining 
\begin{align*}
 \psi_{1,a}(\epsilon_1,r_1,u_0,\mu):=\widetilde \psi_{1,a}(\log( \epsilon_1^{-1} \epsilon_{11}),r_{1},u_0,\mu),
\end{align*}
give the desired expression for the $y_{1,a}$-component of $P_{1,a}$ on the set $V_{1,a}(\chi)$. The function $\psi_{1,a}$ extends to a $C({\epsilon_1}^0(r_1,u_0,\mu)^k)$-function on the closed interval $\epsilon_1\in [0,\epsilon_{10}]$ due to \eqref{phiconv} ($\tau\rightarrow \infty$ as $\epsilon_1\rightarrow 0$). The bound in $C((\epsilon_1,r_1)^0(r_1,u_0,\mu)^k)$ in \eqref{psi1aprop} follows from \eqref{bphi}. Finally, the equality in \eqref{psi1aprop} follows from the invariance of $r_1=y_{1,a}=0$ for \eqref{eqnsredMa1new} (due to the existence of \eqref{gamma2}), see also the second equality in \eqref{bphi}.

We subsequently use the transition time $\tau=\log (\epsilon_1^{-1} \epsilon_{11})$ to exponentially estimate the $z_{1,a}$-component of $P_{1,a}$. This is standard and can be done in $C^k$ with respect to $(r_1,y_{1,a},\mu)$ through variational equations. This completes the proof.
\end{proof}

We can do precisely the same thing for $P_{1,r}$ by working in backward time. We therefore state the following results without proof.
\begin{lemma}\lemmalab{transformnearM1r}
 Consider \eqref{barxneg1eqns} near $z_1=1$ and fix any $k\in \mathbb N$. Then for $\sigma>0$ small enough, there exists a regular transformation of time and a $C^{k+3}$-smooth diffeomorphism $(y_1,z_1)\mapsto (y_{1,r},z_{1,r})$, $y_1,z_1-1\in [-\sigma,\sigma]$, depending $C^{k+3}$-smoothly on $(\epsilon_1,r_1,\mu)$ for $0\le r_1,\epsilon_1,\vert \mu\vert \le \sigma$, such that 
 \begin{equation}\eqlab{y1rz1r}
 \begin{aligned}
  y_{1,r} &=h_{1,r}^0(\epsilon_1,r_1,y_1,\mu)+\epsilon_1 h_{1,r}^1(\epsilon_1,r_1,y_1,z_1,\mu),\\
  z_{1,r}&= z_1-m_{1,r}(\epsilon_1,r_1,y_1,\mu),
 \end{aligned}
 \end{equation}
 where
 \begin{align}\eqlab{h1r0prop}
 h_{1,r}^0(\epsilon_1,0,\mu,\mu)\equiv 0, \quad \frac{\partial h_{1,r}^0}{\partial y_1} (\textnormal{\textbf{0}})=1,\quad \frac{\partial h_{1,r}^0}{\partial \mu} (\textnormal{\textbf{0}})=-1,\end{align}  and \eqref{barxneg1eqns} becomes
 \begin{equation}\eqlab{eqnsredMr1new2}
\begin{aligned}
 \dot r_1 &=\frac12 r_1 \epsilon_1,\\
 \dot{y}_{1,r} &=\left((-\lambda(\mu) +r_1 L_0(r_1,y_{1,r},\mu) )y_{1,r}+r_1 \epsilon_1 L_1(\epsilon_1,r_1,y_{1,r},\mu)\right)\epsilon_1,\\
 \dot{z}_{1,r} &= (2+L_2(r_1,y_{1,r},z_{1,r},\epsilon_1,\mu))z_{1,r},\\
 \dot \epsilon_1 &=-\epsilon_1^2,
\end{aligned}
\end{equation}
with $L_0$, $L_1$, $L_2$ each $C^k$-smooth and $L_2(\textnormal{\textbf{0}})=0$.
 
\end{lemma}
We can as above extend the normal form transformation to $z_1 \in [0,1+\sigma]$, say, by the flow-box theorem. 

Let $\widetilde m_{1,r}(\epsilon_1,r_1,y_{1,r},\mu) =m_{1,r}(r_1,y_1(\epsilon_1,r_1,y_{1,r},\mu),\epsilon_1,\mu)$ with the smooth function $y_1(\epsilon_1,r_1,y_{1,r},\mu)$ being obtained (by the implicit function theorem) from \eqref{y1rz1r} with $z_1=0$:
\begin{align*}
 y_{1,r}=h_{1,r}^0(\epsilon_1,r_1,y_1,\mu)+\epsilon_1 h_{1,r}^1(\epsilon_1,r_1,y_1,0,\mu).
\end{align*}
\begin{proposition}\proplab{P1r}
  Fix $k\in \mathbb N$ and any $\delta>0$ and consider \eqref{eqnsredMr1new2} with the right hand side being $C^{k+1}$. Then there exist constants $c>0$, $\chi>0$ and $0<\epsilon_{10}\ll \epsilon_{11}$, $r_{10}>0$ such that the following holds: The transition map $P_{1,r}(\cdot,\mu)$ of \eqref{eqnsredMr1new2} from $\{z_{1,r}=-\widetilde m_{1,r}(\epsilon_1,r_1,y_{1,r},\mu)\}$ to $\{\epsilon_1=\epsilon_{11}\}$ is well-defined on a region $V_{1,r}(\chi)$ defined by
 \begin{align}
    \vert y_{1,r}\vert \le \chi (\epsilon_1 \epsilon_{11}^{-1})^{\lambda(\mu)},\epsilon_1 \in [0,\epsilon_{10}],r_1\in [0,r_{10}],\eqlab{V1r}
 \end{align}
 and on this region, the mapping takes the following form
   \begin{align*}
  P_{1,r}(\epsilon_1,r_1,y_{1,r},\mu) =\begin{pmatrix}
                                    \sqrt{\epsilon_1 \epsilon_{11}^{-1}}r_1\\
                                    (\epsilon_{11} \epsilon_{1}^{-1})^{\lambda(\mu)} y_{1,r}+\psi_{1,r}(\epsilon_1,r_{1},(\epsilon_{1}^{-1} \epsilon_{11})^{\lambda(\mu)} y_{1,r},\mu)\\
                                    \mathcal O(e^{-c/\epsilon_1})\\
                                    \epsilon_{11}
                                   \end{pmatrix}.
 \end{align*}
 Here $\psi_{1,r}$ is a $C({\epsilon_1}^0(r_1,u,\mu)^k)$-function on the set $\epsilon_1\in [0,\epsilon_{10}],r_1\in [0,r_{10}],0\le \vert u\vert,\vert \mu\vert \le \chi$ with 
 \begin{align}
  \Vert \psi_{1,r}\Vert_{(\epsilon_1,r_1)^0(u,\mu)^k}\le \delta,\quad \psi_{1,r}(\epsilon_1,0,0,\mu)\equiv 0.\eqlab{psi1rprop}
 \end{align}
The $\mathcal O(e^{-c/\epsilon_1})$ remainder term in the $z_{1,r}$-component is also a $C((\epsilon_1,r_1)^0(y_{1,r},\mu)^k)$-function with the order being unchanged upon differentiation up to order $k$ with respect to $(y_{1,r},\mu)$ in $V_{1,r}(\chi)$.
 \end{proposition}


\begin{remark}
 Notice that if we define $\hat y_{1,a}$ and $\hat y_{1,r}$ by 
\begin{align}\eqlab{haty1a}
y_{1,i}= (\epsilon_1 \epsilon_{11}^{-1})^{\lambda(\mu)} \hat y_{1,i}, \quad i=a,r
\end{align} 
for $\epsilon_1\in (0,\epsilon_{10}]$ then $V_{1,i}$ and $V_{1,r}$ are ``blown up'' to $\hat y_{1,i}\in [-\chi,\chi]$, $i=a,r$, and the $y_{1,i}$-components of $P_{1,i}$
\begin{align*}
 \hat y_{1,i} +\psi_{1,i} ( r_{1},\epsilon_1,\hat y_{1,i},\mu),
\end{align*}
become regular as functions of $(r_1,\hat y_{1,i},\mu)$ for $i=a,r$. We will use a similar -- but slightly different scaling -- later on. The reason why \eqref{haty1a} will not work directly, is that we also have to transform $\mu$ appropriately in order to relate $y_{1,a}$ and $y_{1,r}$ (through their relationship to $y_1$) in such a way that the application of $P_{1,a}$ and $P_{1,r}$ correspond to flowing the same point forward and backward in time.
\end{remark}

\subsection{Analysis in the $\bar \epsilon=1$-chart}
Having now followed points on $\{z_1=0\}$ forwards and backwards until the section $\{\epsilon_1=\epsilon_{11}\}$, we proceed to extend these further into the scaling chart. Here the equations are given by \eqref{eqn2}. Notice specifically, that $\epsilon_1=\epsilon_{11}>0$ in the $\bar x=-1$-chart corresponds to
$x_2=-\epsilon_{11}^{-1}$ in the scaling chart.

Let $P_{2,a}$ and $P_{2,r}$ denote the mappings from $\{x_2=-\epsilon_{11}^{-1}\}$ to $\{z_2=0\}$ near $\gamma_{2,0}(0)$, obtained by the first intersection upon application of the forward respectively backward flow of \eqref{eqn2}. Each of these mappings are regular, i.e. smooth diffeomorphisms. 

Through the flow of \eqref{eqn2} we can also extend the center manifolds $M_{2,a}$ and $M_{2,r}$, being the coordinate transformations of $M_{1,a}$ and $M_{1,r}$, respectively, using \eqref{cc12}. The manifolds $M_{2,a}$ and $M_{2,r}$ within the $(x_2,y_2,z_2,r_2)$-space are foliated by constant values of $r_2\sim 0$. Let $M_{2,a}(r_2)$ and $M_{2,r}(r_2)$ denote the corresponding leafs of this foliation projecting onto the $(x_2,y_2,z_2)$-space. It is standard, see \cite{szmolyan_canards_2001}, that $M_{2,a}(0)$ and $M_{2,r}(0)$ intersect transversally along $\gamma_{2,0}(\mu)$ for all $\mu\sim 0$. This gives rise to a connecting orbit $\gamma_{2,r_2}(\mu)\subset M_{2,a}(r_2)\cap M_{2,r}(r_2)$ for all $0<r_2 = \sqrt \epsilon\ll 1$, see \cite{szmolyan_canards_2001}. 
In this way, one may obtain the perturbed (maximal) strong canard $\gamma_{\epsilon}(\mu)$ with $\lim_{\epsilon\rightarrow 0}\gamma_{\epsilon}(\mu)\rightarrow \gamma_0(\mu)$, see \cite{szmolyan_canards_2001}.

\subsection{Putting it all together}
We summarize the local findings in charts into a global diagram in \figref{blowup0}. Notice that the strong canard orbit $\gamma_0$ (in orange) is a heteroclinic orbit of partially hyperbolic points on the blowup sphere. For $\mu=0$, there is a fast jump (in black) that together with $\gamma_0$ gives rise to heteroclinic cycle. It is the perturbation of this cycle, that produce the intermediate periodic orbits that connect to the small ones (as perturbations of the green orbits, recall \figref{fig:x2z2}) and canard cycles \responsenew{(recall \lemmaref{lemma0})}.
\begin{figure}[h!]
\begin{center}
{\includegraphics[width=.53\textwidth]{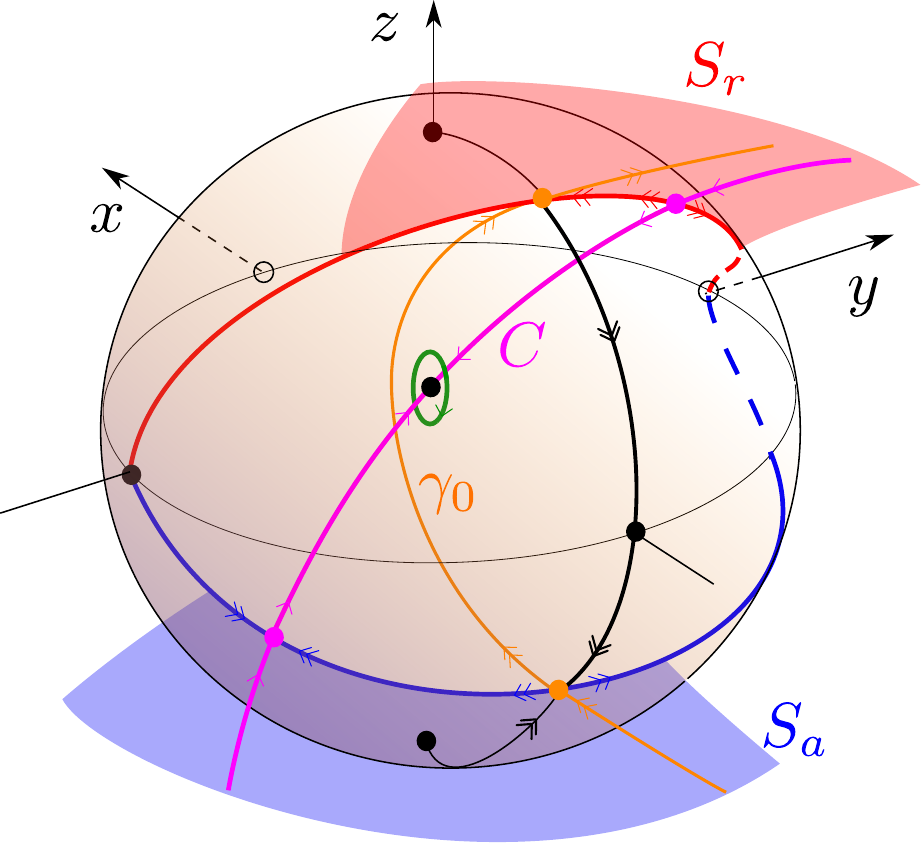}}
\end{center}
\caption{Dynamics for the blown up system for $\mu=0$. The blue and red surfaces are the attracting and repelling critical manifolds. Upon blowup these manifolds extend onto the blowup sphere as normally hyperbolic invariant sets (the red and blue curves, respectively, on the sphere). The weak canard is in purple. In the case of the folded saddle-node for $\mu=0$, it extends onto the blowup sphere as a critical manifold for $\mu=0$. The associated reduced problem has a stable node (indicated by the purple arrows and the black dot inside the sphere), which undergoes a Hopf bifurcation of the layer problem. This Hopf bifurcation is degenerate (the layer problem is Hamiltonian) and there is a family of periodic orbits (one example is shown in green). This family is bounded by the strong canard $\gamma_0$ (in orange). On the blowup sphere, $\gamma_0$ and the fast connection in black from the repelling set (in red) to the attracting set (in blue), produces a heteroclinic cycle. }
\figlab{blowup0}
\end{figure}

We now define $\Delta_{a}(\epsilon_1,r_1,y_{1,a},\mu)$ in the following way:
\begin{align}
 \Delta_a(\epsilon_1,r_1,y_{1,a},\mu) =\left(\Pi \circ P_{2,a} \circ \widehat C_{21} \circ P_{1,a}\right)(\epsilon_1,r_1,y_{1,a},\mu),\eqlab{Deltaa}
\end{align}
with $\Pi$ the projection onto the $(x_2,y_2)$-plane and where $\widehat C_{21}$ is the change of coordinates from $(\epsilon_1,r_1,y_{1,a},z_{1,a})$ to $(x_2,y_2,z_2,r_2)$ defined by \eqref{y1az1a} and \eqref{cc12}. \responsenew{$\widehat C_{12}$ is regular on $\{\epsilon_1=\epsilon_{11}\}$}. $\Delta_a$ therefore describes the transition map from $\{z_1=0\}$ in chart $\bar x=-1$, using the  coordinates $(\epsilon_1,r_1,y_{1,a})$, to $\{z_2=0\}$ in chart $\bar \epsilon=1$, using the coordinates $(x_2,y_2)$ for $r_2=\sqrt{\epsilon}=r_1\sqrt{\epsilon_1}$. 
We define $\Delta_r$ completely analogously:
\begin{align}
 \Delta_r(\epsilon_1,r_1,y_{1,r},\mu) =\left( \Pi \circ P_{2,r} \circ \widehat C_{21} \circ P_{1,r}\right)(\epsilon_1,r_1,y_{1,r},\mu).\eqlab{Deltar}
\end{align}

If $(\epsilon_1,r_1,y_{1,a})\in V_{1,a}(\chi)$ and $(\epsilon_1,r_1,y_{1,r})\in V_{1,r}(\chi)$, under the coordinate transformations \eqref{y1az1a} and \eqref{y1rz1r} for $z_1=0$, correspond to the same point $(\epsilon_1,r_1,y_1,0)$ on the section $\{z_1=0\}$ in the $\bar x=-1$-chart, then $\Delta_a(\epsilon_1,r_1,y_{1,a},\mu)=\Delta_r(\epsilon_1,r_1,y_{1,r},\mu)$ for $r_1\ge 0$, $\epsilon_1>0$ clearly implies existence of a closed orbit of the system(\eqref{fast},$\dot \epsilon=0$) through the point given by $(\epsilon_1,r_1,y_1,0) \in \{z_1=0\}$ in the $\bar x=-1$-chart. 

Solving \eqref{y1az1a} and \eqref{y1rz1r} with $z_1=0$ for $y_1$, we find that
\begin{align*}
 y_1 = \overline h_{1,a}(\epsilon_1,r_1,y_{1,a},\mu),
\end{align*}
and
\begin{align*}
 y_1 =\overline h_{1,r}(\epsilon_1,r_1,y_{1,r},\mu),
\end{align*}
by applying the inverse/implicit function theorem.
Here $\overline h_{1,i}$ have the same smoothness properties as $h_{1,i}$, $\overline h_{1,i}( \epsilon_1,\textnormal{\textbf 0}) \equiv 0,$ $i=a,r$, 
and 
\begin{align}
 \frac{\partial \overline h_{1,a}}{\partial y_{1,a}}(\textbf 0) = 1,&\quad \frac{\partial \overline h_{1,a}}{\partial \mu}(\textbf 0) = -1,\\
 \frac{\partial \overline h_{1,r}}{\partial y_{1,r}}(\textbf 0) = 1,&\quad \frac{\partial \overline h_{1,r}}{\partial \mu}(\textbf 0) = 1.
\end{align}
Equating these expressions for $y_1$, we obtain the following equation
\begin{align}
 K(\epsilon_1,r_1,y_{1,a},y_{1,r},\mu):=\overline h_{1,a}(\epsilon_1,r_1,y_{1,a},\mu)- \overline h_{1,r}(\epsilon_1,r_1,y_{1,r},\mu)=0,\eqlab{Ldef}
\end{align}
relating $y_{1,a}$ and $y_{1,r}$, as desired.
We can use the implicit function theorem to solve this equation for $y_{1,r}$ as function of $(\epsilon_1,r_1,y_{1,a},\mu)$. Subsequently, since we are interested in applying both $P_{1,a}$ and $P_{1,r}$, we need to ensure that the image of the region $V_{1,a}(\chi)$ (or an appropriate subset hereof) under the associated mapping $(\epsilon_1,r_1,y_{1,a})\mapsto (\epsilon_1,r_1,y_{1,r})$ defined by \eqref{Ldef}, will be contained within $V_{1,r}(\chi)$. For this, we therefore first adjust $\mu$.

\begin{lemma}\lemmalab{lemma1}
 There exists a continuous function $\mu_0$ of $(\epsilon_1,r_1)\in [0,\epsilon_{10}]\times [0,r_{10}]$ with $\mu_0(\epsilon_1,0)\equiv 0$ and \responsenew{$\mu_0(0,r_1)\equiv\overline \mu_0(r_1)$, recall \lemmaref{lemma0}}, such that $$K(\epsilon_1,r_1,0,0,\mu_0(\epsilon_1,r_1))=0,$$ for all $\epsilon_1\in [0,\epsilon_{10}]$, $r_1\in [0,r_{10}]$. 
\end{lemma}
\begin{proof}
 Simple application of the implicit function theorem using $K(\epsilon_1,\textbf 0)=0$ and $K'_{\mu}(\textbf 0)=-2$. Moreover, we clearly have $K(0,r_1,0,0,\mu)\equiv K_0(r_1,\mu)$, recall \eqref{L0func}.
\end{proof}
\begin{remark}
 Clearly, $\mu_0\in C^k$ but we will not use such smoothness properties in the proof \thmref{main}. This relates to the fact that the statement of \thmref{main} is only that $\overline \mu(\epsilon,h)$ is continuous. In the author's opinion, the smoothness properties of this function with respect to $h$ is more involved and certainly more difficult to state. 
\end{remark}

Following on from this, we then obtain:
\begin{lemma}\lemmalab{hra}
The equation 
\begin{align*}
  K(\epsilon_1,r_1,y_{1,a},y_{1,r},\mu_0(\epsilon_1,r_1)+\tilde \mu)=0,
\end{align*}
has a locally unique solution of the form
\begin{align}
 y_{1,r} = h_{ra}(\epsilon_1,r_1,y_{1,a},\tilde \mu),\eqlab{hrasol}
\end{align}
with $h_{ra}\in C((\epsilon_1,r_1)^0(y_{1,a},\tilde \mu)^k)$ and
\begin{align}\eqlab{hraprop}
h_{ra}(\epsilon_1,r_1,0,0)\equiv 0, \quad \frac{\partial h_{ra}}{\partial y_{1,a}}(\textnormal{\textbf{0}})=1,\quad \frac{\partial h_{ra}}{\partial \tilde \mu}(\textnormal{\textbf{0}})=-2.
\end{align}
\end{lemma}
\begin{proof}
 Follows from the implicit function theorem, using $$K(\epsilon_1,r_1,0,0,\mu_0(\epsilon_1,r_2))=0,\quad K'_{y_{1,r}}(\epsilon_1,r_1,0,0,\mu_0(\epsilon_1,r_1))\approx 1.$$
\end{proof}

Due to \eqref{hraprop} and the mean value theorem, $h_{ra}(\epsilon_1,r_1,y_{1,a},\tilde \mu ) = y_{1,a}(\cdots)+\tilde \mu(\cdots)$ 
which leads to the following.
\begin{lemma}\lemmalab{hathra}
Let $\sigma>0$, $0<\chi_i<\chi$ for $i=a,r$.
Define the scaled quantities $\hat y_{1,a},\hat y_{1,a}, \hat \mu$ by
\begin{align}
 y_{1,a} = (\epsilon_1 \epsilon_{11}^{-1})^{\lambda(\mu_0)}\hat y_{1,a},\quad y_{1,r} = (\epsilon_1 \epsilon_{11}^{-1})^{\lambda(\mu_0)}\hat y_{1,r},\quad \mu = \mu_0(\epsilon_1,r_1)+ (\epsilon_1 \epsilon_{11}^{-1})^{\lambda(\mu_0)}\hat \mu,\eqlab{hatexpressions}
\end{align}
for $\epsilon_1>0$. 
Then we have the following.
\begin{enumerate}
\item \label{item1} For $\epsilon_{10}>0$ small enough, with $y_{1,i}$ given by \eqref{hatexpressions} for
\begin{align}\eqlab{hatyarsets}
 \hat y_{1,a}\in [-\chi_a,\chi_a],\quad \hat y_{1,r}\in [-\chi_r,\chi_r],\quad \hat \mu\in [-\sigma,\sigma],
\end{align}
we have $(\epsilon_1,r_1,y_{1,i})\in V_{1,i}(\chi)$, $i=a,r$ for all $\epsilon_1\in (0,\epsilon_{10}]$. 
\item \label{item2} \eqref{hrasol} becomes
\begin{align*}
 \hat y_{1,r} = \hat h_{ra}(\epsilon_1,r_1,\hat y_{1,a},\hat \mu):=(\epsilon_1 \epsilon_{11}^{-1})^{-\lambda(\mu_0)}h_{ra}(\epsilon_1,r_1,(\epsilon_1 \epsilon_{11}^{-1})^{\lambda(\mu_0)}\hat y_{1,a},(\epsilon_1 \epsilon_{11}^{-1})^{\lambda(\mu_0)}\hat \mu),
\end{align*}
for $\epsilon_1\in (0,\epsilon_{10}]$. $\hat h_{ra}$ has a  $C((\epsilon_1,r_1)^0(\hat y_{1,a},\hat \mu)^{k-1})$-extension to the closed interval $\epsilon_1\in [0,\epsilon_{10}]$ with $\hat h_{ra}(\epsilon_1,r_1,0,0)\equiv 0$. 
\item \label{item3}
 $\hat h_{ra}(0,r_1,\cdot,\cdot)$ is linear and for $r_1=0$:
\begin{align*}
\hat h_{ra}(0,0,\hat y_{1,a},\hat \mu)=\hat y_{1,a}-2\hat \mu.
\end{align*}
\item \label{item4} For $\sigma>0$, $r_{10}>0$ and $\epsilon_{10}>0$ all small enough, we have that:
\begin{align*}
 \hat h_{ra}(\epsilon_1,r_1,\hat y_{1,a},\hat \mu)\in (-\chi,\chi),
\end{align*}
for all $\epsilon_1\in [0,\epsilon_{10}],r_1\in [0,r_{10}], \hat y_{1,a}\in [-\chi_a,\chi_a], \hat \mu\in [-\sigma,\sigma]$. 

\end{enumerate}

\end{lemma}
\begin{proof}
 Recall from \eqref{V1a} and \eqref{V1r} that $(\epsilon_1,r_1,y_{1,i})\in V_{1,i}(\chi)$ with $\epsilon_1>0$ if and only if
 \begin{align}
  (\epsilon_1^{-1} \epsilon_{11})^{\lambda(\mu)}\vert y_{1,i}\vert \le \chi.\eqlab{V1aragain}
 \end{align}
Now, by \eqref{lambdamu} and  \eqref{hatexpressions} for $\vert\hat y_{1,i}\vert \le \chi_i$ the left hand side of \eqref{V1aragain} becomes
\begin{align}
(\epsilon_1\epsilon_{11}^{-1})^{(\epsilon_1 \epsilon_{11}^{-1})^{\lambda(\mu_0(\epsilon_1,r_1))} \hat \mu} \vert \hat y_{1,i} \vert \le e^{(\epsilon_1 \epsilon_{11}^{-1})^{\lambda(\mu_0(\epsilon_1,r_1))} \hat \mu \log (\epsilon_1 \epsilon_{11}^{-1})}\chi_{i} = (1+\mathcal O(\epsilon_1^{\lambda(\mu_0)}\log \epsilon_1))\chi_i,\eqlab{estimatepseps}
\end{align}
and the item \ref{item1} therefore follows. Items \ref{item2} and \ref{item3} follow from \lemmaref{hra}; notice we lose one degree of smoothness due to the application of the mean value theorem. Finally, item \ref{item4} follows from \ref{item2} and \ref{item3} using $0<\chi_a<\chi$ and $\epsilon_{10}>0$ small enough.
\end{proof}


The implications of item \ref{item1}.  are the following:
Let 
\begin{equation}\eqlab{widehatP1i}
\begin{aligned}
 \widehat P_{1,i}(\epsilon_1,r_1,\hat y_{1,i},\hat \mu) &:= P_{1,i}(\epsilon_1,r_1,y_{1,i},\mu),
\end{aligned}
\end{equation}
for $i=a,r$ and 
with $y_{1,a}$, $y_{1,r}$ and $\mu$ on the right hand side given by the expressions in \eqref{hatexpressions}. Then we have:
\begin{lemma}\lemmalab{widehatPlemma}
Fix $\delta>0$ small enough. Then for $r_{10}>0$, $\epsilon_{10}>0$ sufficiently small, $$\widehat P_{1,i}\in C((\epsilon_1,r_1)^0(\hat y_{1,i},\hat \mu)^{k}),$$ $i=a,r$, defined for $\hat y_{1,a}\in [-\chi_a,\chi_r],\hat y_{1,r}\in [-\chi_r,\chi_r]$, \responsenew{respectively,} and $\hat \mu\in [-\sigma,\sigma]$,  $\epsilon_1\in [0,\epsilon_{10}]$, $r_1\in [0,r_{10}]$. 
In particular, the $y_{1,i}$ component of $\widehat P_{1,i}$ takes the following form:
\begin{align}
 \hat y_{1,i}+\widehat \psi_{1,i}(\epsilon_1,r_1,\hat y_{1,i},\hat \mu),\eqlab{expansionP1y}
\end{align}
where 
\begin{align}\eqlab{widehatpsi}
\Vert \widehat \psi_{1,i}\Vert_{(\epsilon_1,r_1)^0(\hat y_{1,i},\hat \mu)^k}\le \delta,\quad \widehat \psi_{1,i}(\epsilon_1,0,0,\hat \mu)\equiv 0.\end{align}
\end{lemma}
\begin{proof}
 We insert \eqref{hatexpressions} into the $y_{1,i}$-component of $P_{1,i}$. Proceeding as in \eqref{estimatepseps} for the expansion in $\epsilon_1$ we obtain \eqref{expansionP1y}. \eqref{widehatpsi} follows from \eqref{psi1aprop} and \eqref{psi1rprop}.
\end{proof}

Moreover, by item \ref{item4}. we have that the composed mapping
\begin{align}
(\epsilon_1,r_1,\hat y_{1,a},\hat \mu) \mapsto \widehat P_{1,r}(\epsilon_1,r_1,\hat h_{ra}(\epsilon_1,r_1,\hat y_{1,a},\hat \mu),\hat \mu),
\end{align}
is a $C((\epsilon_1,r_1)^0(\hat y_{1,a},\hat \mu)^{k-1})$-function
defined for $y_{1,a}\in [-\chi_a,\chi_a]$ and $\hat \mu\in [-\sigma,\sigma]$  and all $r_1\in [0,r_{10}]$, $\epsilon_1\in [0,\epsilon_{10}]$ for $\sigma>0$ and $\epsilon_{10}>0$ small enough.

We now return to the problem of solving $\Delta_a=\Delta_r$ for closed orbits, \responsenew{recall \eqref{Deltaa} and \eqref{Deltar}}. 
Define
 $$\widehat \Delta_i(\epsilon_1,r_1,\hat y_{1,i},\hat \mu)=\Delta_i(\epsilon_1,r_1,y_{1,i},\mu),\quad i=a,r$$ with $y_{1,i}$ and $\mu$ on the right hand side given by the expressions \eqref{hatexpressions}.
By the properties of $\widehat P_{1,i}$ and the regularity of $C_{21}$ and $P_{2,i}$, we have
$\widehat \Delta_{i}\in C((\epsilon_1,r_1)^0(\hat y_{1,i},\hat \mu)^{k})$ for $i=a,r$. 
Finally, let 
\begin{align}\eqlab{widehatDelta}
 \widehat \Delta(\epsilon_1,r_1,\hat y_{1,a},\hat \mu ) &:=\widehat  \Delta_r\left(\epsilon_1,r_1,\hat h_{ra}(\epsilon_1,r_1,\hat y_{1,a},\hat \mu),\hat \mu\right)-\widehat \Delta_a\left(\epsilon_1,r_1,\hat y_{1,a},\hat \mu\right),
\end{align}
for $ \epsilon_1\in [0,\epsilon_{10}],r_1\in [0,r_{10}], \hat y_{1,a}\in [-\chi_a,\chi_a], \hat \mu\in [-\sigma,\sigma]$. 
By \lemmaref{hathra} item \ref{item2}, we finally have
that $\widehat \Delta \in C((\epsilon_1,r_1)^0(\hat y_{1,a},\hat \mu)^{k-1})$.

Clearly, 
\begin{align}\eqlab{DeltaEq}
\widehat \Delta(\epsilon_1,r_1,\hat y_{1,a},\hat \mu)=\textnormal{\textbf{0}},
\end{align}defines closed orbits for $\epsilon_1>0$, $r_1\ge 0$ sufficiently small. 
We therefore proceed to solve this equation. We will do so by solving for $\hat y_{1,a}$, $\hat \mu$ as functions of $\epsilon_1$ and $r_1$ using the implicit function theorem. Henceforth we therefore take $k\ge 2$. 
\begin{lemma}\lemmalab{Delta0prop}
\begin{align}\eqlab{WidehatDelta0}
  \widehat \Delta(\textnormal{\textbf 0})=\textnormal{\textbf 0}.
 \end{align}
Moreover, there exists two nonzero tangent vectors $$v_i=(*,*,0)\in T_{\gamma_{2,0}(0)\cap \{z_2=0\}} M_{2,i}(0),$$ $i=a,r$, such that: 

\begin{align}
 \frac{\partial \widehat \Delta}{\partial \hat y_{1,a}} (\textnormal{\textbf 0}) &=\Pi v_r-\Pi v_a,\nonumber\\
 \frac{\partial \widehat \Delta}{\partial \hat \mu} (\textnormal{\textbf 0}) &=-2\Pi v_r.\eqlab{hatDeltaHatMu}
\end{align}
\end{lemma}
\begin{proof}
First, \eqref{WidehatDelta0} follows from the connection $\gamma_{2,0}(0)$. Next, for the proof of the the partial derivatives, we first focus on the partial derivatives of $\widehat \Delta_a$.  For this we differentiate $\widehat P_{1,a}$. Following \lemmaref{widehatPlemma}, see \eqref{widehatpsi}, we have $\frac{\partial \widehat P_{1,a}}{\partial \hat \mu}(\textbf 0)=\textbf 0$, and hence $\frac{\partial \widehat \Delta_a}{\partial \hat \mu}(\textbf 0)=\textbf 0$. 
Next, by \eqref{expansionP1y}
\begin{align}
 \frac{\partial \widehat P_{1,a}}{\partial \hat y_{1,a}}(\textbf 0) = \begin{pmatrix}
                                                                       0\\
                                                                       0\\
                                                                       1+\mathcal O(\delta)\\
                                                                       0
                                                                      \end{pmatrix}.\eqlab{partialwidehatPy1}
\end{align}
Since $\widehat P_{1,a}(\textbf 0)=(\epsilon_{11},0,0,0)\in M_{1,a}$, the vector \eqref{partialwidehatPy1} gives a tangent vector to $M_{2,a}(0)$ at $x_2=-\epsilon_{11}^{-1}$ upon application of the tangent map $T_{(\epsilon_{11},0,0,0)}\widehat C_{21}$ of the change of coordinates $\widehat C_{21}$. Consequently, by applying $T P_{2,a}$ we have
\begin{align*}
 \frac{\partial \widehat \Delta_a}{\partial \hat y_{1,a}}(\textbf 0)=\Pi v_a,
\end{align*}
for some nonzero tangent vector $v_a\in T_{\gamma_{2,0}(0)\cap \{z_2=0\}} M_{2,a}(0)$.

$\widehat \Delta_r$ can be handled similarly: 
\begin{align*}
 \frac{\partial \widehat \Delta_r}{\partial \hat y_{1,r}}(\textbf 0)=\Pi v_r,\quad \frac{\partial \widehat \Delta_r}{\partial \hat \mu}(\textbf 0)=\textbf 0
\end{align*}
for some nonzero tangent vector $v_r\in T_{\gamma_{2,0}(0)\cap \{z_2=0\}} M_{2,r}(0)$.
The partial derivative \eqref{hatDeltaHatMu} with respect to $\hat \mu$ then follows from \lemmaref{hathra}, see item \ref{item3}, and the chain rule.

\end{proof}
Since the intersection of $M_{a,2}(0)$ and $M_{r,2}(0)$ is transverse along $\gamma_2$, the vectors $v_i$, $i=a,r$, as well as $\Pi v_i$, $i=a,r$, are linearly independent and consequently, 
\begin{align*}
 \text{det}\left(\frac{\partial \widehat \Delta}{\partial (\hat y_{1,a},\hat \mu)}(\textnormal{\textbf 0})\right) =\text{det}\begin{pmatrix} \Pi v_r-\Pi v_a & -2\Pi v_r \end{pmatrix} = 2\text{det}\begin{pmatrix} \Pi v_a&\Pi v_r \end{pmatrix}\ne 0.
\end{align*}
 In this way, by the implicit function theorem we can solve $\widehat \Delta(\epsilon_1,r_1,\hat y_{1,a},\hat \mu)=0$ locally for $(\hat y_{1,a},\hat \mu)$ as continuous functions of $(\epsilon_1,r_1)$. In this way, we conclude the following.
\begin{proposition}\proplab{prop:large}
There exist continuous functions $\bar y_1(\epsilon_1,r_1)$, $\bar \mu_1(\epsilon_1,r_1)$, $\epsilon_1\in [0,\epsilon_{10}],r_1\in [0,r_{10}]$ for $\epsilon_{10}>0$ and $r_{10}>0$ small enough, with $\bar y_1(0,0)=\bar \mu_1(0,0)=0$ such that there is a periodic orbit of (\eqref{fast},$\dot \epsilon=0$) through any point $(\epsilon_1,r_1, \bar y_1(\epsilon_1,r_1),0)$ in chart $\bar x=-1$ for $\mu=\bar \mu_1(\epsilon_1,r_1)$ for all $\epsilon_1\in (0,\epsilon_{10}],r_1\in [0,r_{10}]$.
\end{proposition}
\begin{proof}
After having solved $\widehat \Delta=0$ for $\hat y_{1,a}$ and $\hat \mu$ as continuous functions $(\epsilon_1,r_1)$ -- using the implicit function theorem and \lemmaref{Delta0prop} -- we complete the result by mapping the solution back to the $(y_1,\mu)$-variables. 
\end{proof}
\section{Completing the proof of \thmref{main}}\seclab{completing}
To complete the proof of \thmref{main}, we first collect our results thus far: First, by \propref{smallcycles} we have for any $h_0>0$ a family of periodic orbits $\Gamma^{\textnormal{small}}_{2,\epsilon,h}$, $h\in (0,h_0]$, of \eqref{eqn2} with $\mu=r_2 \overline \mu_2(\epsilon,h)$ for all $r_2=\sqrt{\epsilon}>0$ small enough, parametrized in the $\tilde u$-variables, see \eqref{tildeust}, by their intersection $\tilde u=(-h,0)$, $y_2= \overline y_2(h,r_2)$ with $\tilde u_2=0$.

\begin{lemma}\lemmalab{overlap}
Consider \eqref{eqn2} for parameter values $\mu \in [-\delta,\delta]$, and let $D$ be any compact domain in the phase $(x_2,z_2)$-plane. Then there is an $h_0>0$ such that any periodic orbit of this system for $r_2>0$ small enough, that is contained within the set defined by $(x_2,z_2)\in D$, $y_2\in [-\delta,\delta]$ and intersects $\tilde u_1<0$, $\tilde u_2=0$ in a single point, belongs to the family $\Gamma^{\textnormal{small}}_{2,\epsilon,h}$, $h\in (0,h_0]$. 
\end{lemma}
\begin{proof}
 The statement is clearly true for $\mu=r_2\mu_2$ with $\mu_2\in [-\delta,\delta]$ for $\delta>0$ small enough, since the family $\Gamma^{\textnormal{small}}_{2,h,\epsilon}$, $h\in (0,h_0]$, is obtained by the implicit function theorem. Next, we realize that there are no periodic orbits for any $\vert \mu_2\vert>\delta$ and $r_2>0$ small enough since in this case, the equilibrium \eqref{y2eq} is hyperbolic; specifically, $C_2$\response{, recall \eqref{C2set},} is normally hyperbolic in a neighborhood of \eqref{y2eq}. Finally, for $r_2c \le \vert \mu\vert \le \delta$ with $c>0$ large enough, $\dot y_2$ has a single sign in the set $(x_2,z_2)\in D$, $y_2\in [-\delta,\delta]$, completing the proof. 
\end{proof}
Clearly, the family $\Gamma_{2,\epsilon,h}^{\textnormal{small}}$, $h\in (0,h_0]$, becomes a family of ``small'' periodic orbits $\Gamma_{\epsilon,h}^{\textnormal{small}}$ of \eqref{fast} upon blowing down, intersecting $z=0$ in $x\sim \sqrt{\epsilon} h$, $y=r_2 \overline y_2(h,\sqrt{\epsilon})$, $h\in (0,h_0]$. (Here we have used $\sim$ in the expression for $x$ since this only holds to leading order; the family is parametrized by $\tilde u_1=h$ and $x_2$ and $\tilde u_1$ differ by $\mathcal O(\sqrt{\epsilon})$ in the relevant domain.)


Next, we turn to the intermediate periodic orbits obtained from \propref{prop:large}. Here we obtain a family of periodic orbits $\Gamma^{\textnormal{inter}}_{\epsilon,h}$, $h\in[\sqrt{\epsilon}/\sqrt{\epsilon_{11}},r_{10}]$,  of \eqref{fast} with $\mu=\overline \mu_1(h^{-2}\epsilon,h)$ for all $\epsilon>0$ small enough, parametrized by their intersection  $(x,y,z)=(-h^2,h\overline y_1(h^{-2}\epsilon,h),0)$. This follows from \propref{prop:large} and \eqref{barxneg1}, specifically $\epsilon=r_1^2\epsilon_1$ and $r_1=h$. Now, $r_1=h=\sqrt{\epsilon} /\sqrt{\epsilon_{11}}$ corresponds to $\epsilon_1=\epsilon_{11}$ or $x_2=-1/\epsilon_{11}$ upon change of coordinates \eqref{cc12}. Consequently, upon taking $h_0> -1/\sqrt{\epsilon_{11}}$ the branch $\Gamma^{\textnormal{inter}}_{2,\epsilon,h}$, i.e. $\Gamma^{\textnormal{inter}}_{\epsilon,h}$ written in the coordinates of the $\bar \epsilon=1$-chart, overlap with $\Gamma_{2,\epsilon,h}^{\textnormal{small}}$ for all $0<\epsilon\ll 1$. By \lemmaref{overlap}, where they overlap, they coincide for all $0<\epsilon\ll 1$. In this way, we obtain the desired family \response{$\Gamma_{\epsilon,h}$} by gluing the two branches together in the domain where they overlap.  \response{In other words, there are continuous functions $H_0$, $H_1$, $H_2$ with $H_2(0,h)\equiv -h^2$, and $\overline \mu$, such that for each $h\in (0,h_1]$ and all $0<\epsilon\ll 1$ ($h_0>0$ large enough and $h_1>0$ small enough):
\begin{align*}
 \Gamma_{\epsilon,h} = \begin{cases}
                        \Gamma_{H_0(\epsilon,h),\epsilon}^{\textnormal{small}} \quad \text{for} \quad h\in [0,\sqrt{\epsilon} h_0],\\
                        \Gamma_{H_1(\epsilon,h),\epsilon}^{\textnormal{inter}}\quad \text{for} \quad h\in (\sqrt{\epsilon} h_0,h_1],
                       \end{cases}
\end{align*}
is a periodic orbit of \eqref{fast} for $\mu=\overline \mu(\epsilon,h)$, that intersects $z=0$ in $(x,y,z)=(H_2(\epsilon,h),*,0)$. This completes the proof of \thmref{main}.}




\section{Discussion}\seclab{discuss}
Our main theorem generalizes the birth of canard cycles in $\mathbb R^2$, see e.g. \cite{krupa2001a}, to $\mathbb R^3$ through the folded saddle-node of type II and the strong canard. In contrast to the result in \cite{krupa2001a}, our results are -- however -- only local. (The local neighborhood \textit{is} independent of $\epsilon$). Nevertheless, it is also possible to use our method to obtain more global results. We have illustrated a situation in \figref{fig:canardswheads}(a) of an $S$-shaped critical manifold (as in the Koper model \cite{guckenheimer2015a,koper1991a}). Here we can also follow the forward and backward flow of the set of points on the section (in orange) that is indicated in the figure using the normal forms and the solutions of the Shilnikov problem, see \lemmaref{transformnearM1a} and \propref{this}. In contrast to $P_{2,i}$,$i=a,r$, above, we would just define transition maps from $r_1=r_{10}$ fixed to $\epsilon_1=\epsilon_{11}$. Otherwise, the proof for the existence of closed orbits, would proceed analogously, using that the center manifolds $M_{a,r}$ intersect transversally along the strong canard $\gamma_0(\mu)$. (Notice that in general these cycles occur for $\mu=\mathcal O(1)$, i.e. at an order one distance from the folded saddle-node.) The same holds for the perturbation of the singular cycles in \figref{fig:canardswheads}(b). (Having said that, the transition from (a) to (b) is more complicated.) These extensions could even be done in the smooth setting. The  cycles in \figref{fig:canardswheads}(c) mark the end of canard cycles and in $\mathbb R^2$ this is where classical relaxation oscillations appear. However, in contrast to $\mathbb R^2$ the folded singularity $p$ can be of different types, e.g. folded nodes or folded saddles (since we are away from the folded saddle-node in general). 

We proved our main result in the analytic setting. In particular, we used analyticity to prove existence of a slow manifold, that is not normally hyperbolic but acts as the center of (normal) oscillations. This in connection with Melnikov theory, enabled an extension of the Hopf cycles. Although results on normally elliptic slow manifolds  \cite{balser1994a,braaksma1992a,de2020a,kristiansenwulff} are almost exclusively in the analytic setting, it seems plausible (following the work of \cite{bonckaert1986a,bonckaert2005a}, discussed in \secref{main}) that the result could be extended to the smooth setting.  

Finally, we note that our main result does not rely upon analyticity with respect to $\epsilon$, only in the space variables. This suggests that our results may also apply to systems that have been reduced from a slow manifold reduction from a higher dimension. Indeed, slow manifolds in analytic systems have been shown to be analytic in space variables (and only Gevrey in $\epsilon$), see e.g. \cite{de2020a} and references herein. 
\begin{figure}[h!]
\begin{center}
\subfigure[]{\includegraphics[width=.49\textwidth]{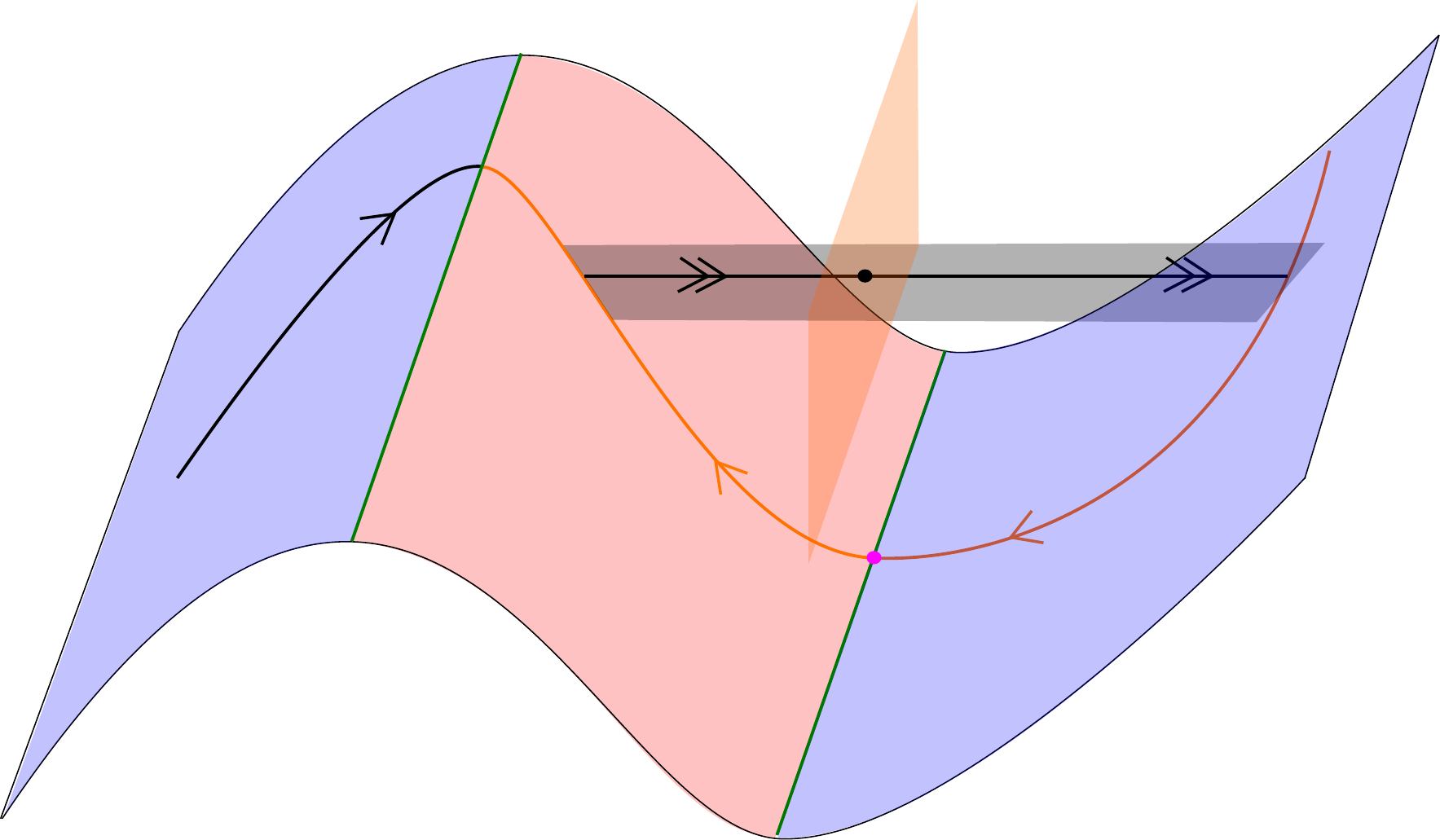}}
\subfigure[]{\includegraphics[width=.49\textwidth]{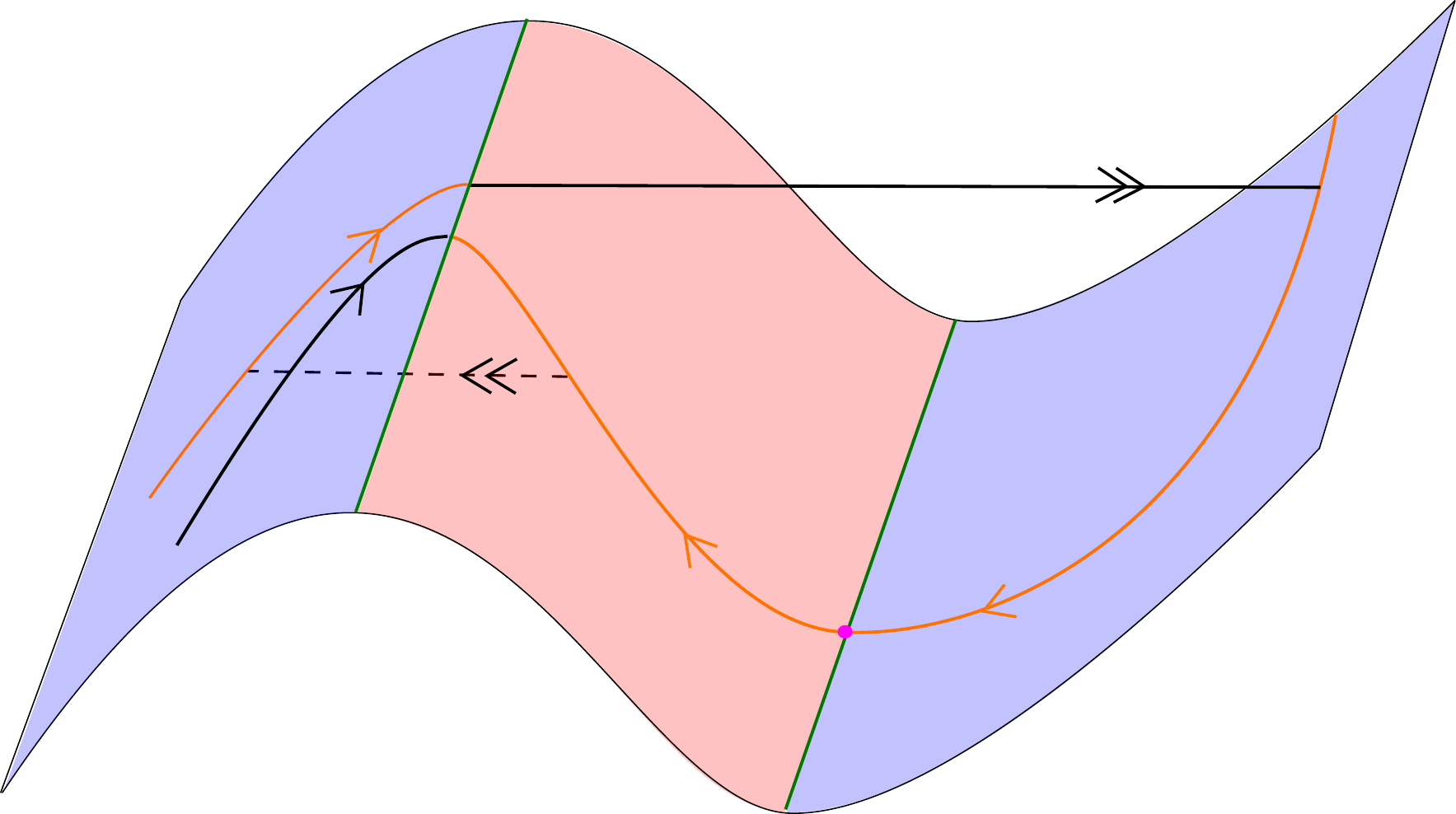}}
\subfigure[]{\includegraphics[width=.49\textwidth]{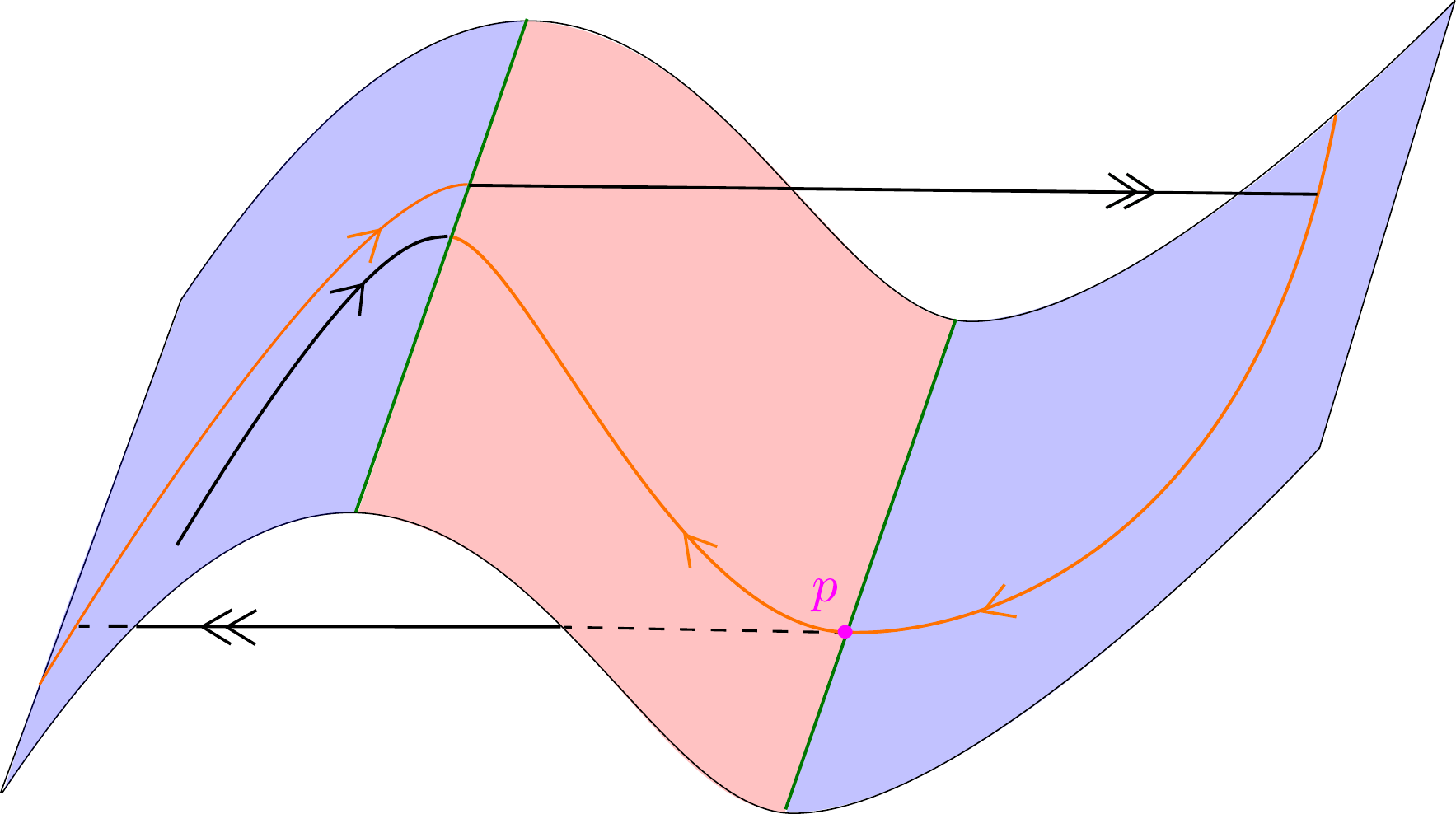}}
\end{center}
\caption{Canard cycles in $\mathbb R^3$ in the case of a $S$-shaped critical manifold. In (a): canard cycle without head (following the terminology from $\mathbb R^2$, see \cite{krupa2001a}), (b): canard cycles with head and (c): the end of the canard cycles, where transition to classical relaxation oscillations appear.}
\figlab{fig:canardswheads}
\end{figure}

\responsenew{In future work, we will study bifurcations of the periodic orbits. In particular, we are interested in a description of the period doubling bifurcations that have been reported in several papers, see e.g. \cite{zaks2011a}. These period doubling bifurcations are associated with the heteroclinic connection on the blowup sphere (recall \figref{blowup0}) for $\epsilon\rightarrow 0$; notice specifically that the period doubling bifurcations cannot occur within compact sets of the $(x_2,y_2,z_2)$-space as the multipliers are closed to $1$ there. We believe that \propref{P1a} and \propref{P1r} provide the adequate details to analyze such bifurcations for all $0<\epsilon\ll 1$. }

\textbf{Acknowledgement}. The author thanks Renato Huzak for providing valuable feedback on earlier versions of the manuscript. 

\bibliography{refs}
\bibliographystyle{plain}
\newpage
\appendix
\section{Proof of \propref{this}}\applab{proofprop32}
We consider \eqref{eqnsredMa1new}, repeated here for convinience:
\begin{equation}\eqlab{eqnsredMa1newapp}
\begin{aligned}
 \epsilon_1' &=\epsilon_1,\\
 r_1' &=-\frac12 r_1,\\
 y_{1,a}' &=(\lambda(\mu) +r_1 L_0(r_1,y_{1,a},\mu) )y_{1,a}+r_1 \epsilon_1 L_1(\epsilon_1,r_1,y_{1,a},\mu),
\end{aligned}
\end{equation}
with 
\begin{align*}
 r_1(t) = e^{-\frac12 t} r_{10},\quad \epsilon_1(t) = e^{(t-\tau)} \epsilon_{11},
\end{align*}
and $y_{1,a}(\tau) = y_{11}$. 
We write 
\begin{align}\eqlab{y1texpr}
y_{1,a}(t)=e^{\lambda(\mu) (t-\tau)} e^{\int_{\tau}^t  r_1(s) L_0(r_1(s),0,\mu)ds} \left(y_{11}+u(t)\right),
\end{align}
with $u(\tau)=0$.
Here 
\begin{align*}
 \vert \int_{\tau}^t  r_1(s) L_0(r_1(s),0,\mu)ds\vert \le  2 r_{10} c_0,
\end{align*}
for $0\le t\le \tau$ using $\int_0^\tau e^{-\frac12 s}ds \le 2$
and $\vert L_0(r_1,0,\mu)\vert \le c_0$ for some $c_0>0$.
Then from \eqref{eqnsredMa1newapp} and variation of constants, we have
 \begin{equation}\eqlab{ut}
\begin{aligned}
 u(t) &= \int_{\tau}^t   r_1(s)   e^{\lambda(\mu) (s-\tau)} e^{\int_{\tau}^s  r_1(v) L_0(r_1(v),0,\mu)dv}  \overline L_0(r_1(s),y_{1,a}(s),\mu) \left(y_{11}+u(s)\right)^2 ds\\
 &+ \int_{\tau}^t r_1(s) \epsilon_1 (s) e^{-\lambda (\mu)(s-\tau)} e^{-\int_{\tau}^s  r_1(s) L_0(r_1(s),0,\mu)ds}  L_1(r_1(s),\epsilon_1(s), y_{1,a}(s),\mu)ds,
\end{aligned}
\end{equation}
with the smooth function $\overline L_0$ defined by $$L_0(r_1,y_{1,a},\mu) = L_0 (r_1,0,\mu)+\overline L_0(r_1,y_{1,a},\mu)y_{1,a},$$
and the mean value theorem.
Let $\alpha>0$. We then consider the closed subset 
\begin{align}
\Gamma_\alpha(\delta)=\{u\in C_\alpha([0,\tau];\mathbb R)\,\vert \Vert u\Vert_\alpha \le \delta,\,u(\tau)=0\},
\end{align}
of the Banach space $C_\alpha([0,\tau];\mathbb R)$ equipped with the norm
\begin{align*}
 \Vert u\Vert_\alpha := \sup_{t\in [0,\tau]} e^{\alpha \tau }\vert u(t)\vert.
\end{align*}
 $y_{1,a}(t)$ in \eqref{y1texpr} is bounded by $2(\chi+\delta)$ for $\vert y_{11}\vert \le \chi$, $u\in \Gamma_\alpha(\delta)$ and all $r_{10}>0$ small enough. 
%
%
%

Let $\Theta(u)(t)$, with $y_{1,a}$ given by \eqref{y1texpr}, denote the right hand side of \eqref{ut}. We then show that $\Theta$ is a contraction on $\Gamma_\alpha(\delta)$, upon choosing the constants appropriately. 

For this purpose, we first emphasize that $\Theta(u)(\tau)=0$. Moreover, by \eqref{lambdamu}, we have $\vert \lambda(\mu)-\frac12 \vert \le \frac12 \chi$ for all $\vert \mu\vert \le \chi$. Then for $\vert \mu\vert \le \chi$, $\vert y_{11}\vert \le \chi$, we
estimate
\begin{align}
 \vert \Theta(u)(t)\vert &\le 
   2 c_1 r_{10}  (\chi+\delta)^2 \int_{0}^\tau   e^{-\frac12 s}  e^{\lambda(\mu) (s-\tau)} ds+ 2c_1 r_{10}\epsilon_{11}  \int_{0}^\tau   e^{-\frac12 s} e^{(s-\tau)}  e^{-\lambda(\mu) (s-\tau)} ds\nonumber \\
 &\le 2 c_1 \tau e^{(\chi-\frac12) \tau} r_{10}   \left((\chi+\delta)^2  +\epsilon_{11}\right),\eqlab{ThetaEst}
\end{align}
for $0\le t\le \tau$, $r_{10}>0$ small enough,
using $\vert \overline L_0(r_1,y_{1,a},\mu)\vert,\vert  L_1(r_1,\epsilon,y_{1,a},\mu)\vert \le c_1$ for some $c_1>0$. 

Moreover, seeing that $\overline L_0$ and $L_2$ are smooth functions, there is a \response{$c_2>0$} such that 
\begin{align*}
 \vert \overline L_0(r_1(t),y_{1,a}(t),\mu)(y_{11}+u(t))^2-\overline L_0(r_1(t),\tilde y_{1,a}(t),\mu) (y_{11}+\tilde u(t))^2 \vert &\le c_2 \vert u(t)-\tilde u(t)\vert,
 \end{align*}
 and
 \begin{align*}
 \vert L_1(r_1(t),\epsilon_1(t), y_{1,a}(t),\mu)-L_1(r_1(t),\epsilon_1(t),\tilde y_{1,a}(t),\mu) \vert &\le c_2 \vert u(t)-\tilde u(t)\vert,
 \end{align*}
for all $0\le t\le \tau$, $\tau>\tau_0$, $u,\tilde u\in \Gamma_{\alpha}(\delta)$ and \response{$c_2>0$}.
This leads to 
\begin{align*}
 \Vert \Theta(u) - \Theta(\tilde u)\Vert_\alpha \le 2\tau c_2e^{(\chi-\frac12) \tau}   r_{10} \left(1+\epsilon_{11}\right)  \Vert u-\tilde u\Vert_\alpha,
\end{align*}
for any $\tilde u,u\in \Gamma_\alpha(\delta)$. 

 We fix $\alpha\in (0,\frac12)$. From the preceding estimates, it then follows that 
\begin{align*}
 \Theta:\Gamma_{\alpha}(\delta)\rightarrow \Gamma_{\alpha}(\delta),
\end{align*}
is well-defined and a contraction
for $\delta>0$, $r_{10}>0$, $\epsilon_{11}>0$ and $\chi>0$ small enough.

We denote the unique fixed-point -- existence of which follows from Banach's fixed point theorem --  by $\underline u(t,\tau,\epsilon_{11},r_{10},y_{11},\mu)$. From \eqref{ThetaEst}, we find that there is a  \response{$c_3>0$} such that
\begin{align*}
 \vert \underline u(t,\tau,\epsilon_{11},r_{10},y_{11},\mu)\vert \le e^{-\alpha \tau} \response{c_3} r_{10},
\end{align*}
for all $0\le t\le \tau$, $\tau>\tau_0$. 
Writing 
\begin{align*}
 \phi(t,\tau,\epsilon_{11},r_{10},y_{11},\mu) &:=  \left(e^{\int_{\tau}^t  r_1(s) L_0(r_1(s),0,\mu)ds}-1\right)y_{11} \\
 &+e^{\int_{\tau}^t  r_1(s) L_0(r_1(s),0,\mu)ds}  \underline u(t,\tau,\epsilon_{11},r_{10},y_{11},\mu),
\end{align*}
gives \propref{this} for $k=0$.
%
%
%
Now,  
regarding the $C^k$-smoothness of $\phi$, we can proceed completely analogously by setting up fixed-point equations for the partial derivatives of $\underline u$. We leave out further details, and refer instead to \cite{deng1988a}. 
 \end{document}